\documentclass[12pt]{article}
\usepackage{amsthm,amsfonts,amsmath,amssymb,latexsym,epsfig}

\title{The Viterbo--Maslov Index in Dimension Two}

\author{
Joel~W.~Robbin\\
University of Wisconsin
\and
\and
Dietmar~A.~Salamon\thanks{Partially supported by the 
Swiss National Science Foundation Grant 200021-127136}\\
ETH-Z\"urich}

\date{2 May 2012}

\newcommand{\MAT}[1]{\left[\begin{array}{#1}}
\newcommand{\RIX}{\end{array}\right]}
\newcommand{\p}{\partial}

\newcommand{\C}{{\mathbb C}}
\newcommand{\D}{{\mathbb D}}

\newcommand{\HH}{{\mathbb H}}
\newcommand{\N}{{\mathbb N}}

\newcommand{\R}{{\mathbb R}}

\newcommand{\Z}{{\mathbb Z}}
\newcommand{\cD}{{\mathcal D}}

\newcommand{\cP}{{\mathcal P}}

\newcommand{\Om}{{\Omega}}

\newcommand{\eps}{{\varepsilon}}
\renewcommand{\phi}{{\varphi}}

\newcommand{\tSi}{{\tilde{\Sigma}}}
\newcommand{\tal}{{\tilde{\alpha}}}
\newcommand{\tbe}{{\tilde{\beta}}}

\newcommand{\tLa}{{\tilde{\Lambda}}}
\newcommand{\tA}{{\tilde{A}}}
\newcommand{\tB}{{\tilde{B}}}

\newcommand{\tf}{{\tilde{f}}}

\newcommand{\tu}{{\tilde{u}}}

\newcommand{\tx}{{\tilde{x}}}
\newcommand{\ty}{{\tilde{y}}}
\newcommand{\tz}{{\tilde{z}}}

\newcommand{\im}{{\rm im }}        % image
\newcommand{\id}{{\rm id}}         % identity

\newcommand{\IM}{{\rm Im }}        % imaginary part
\renewcommand{\Re}{{\rm Re}}       % real part
\newcommand{\Diff}{{\rm Diff}}        % Diffeomorphisms
\newcommand{\w}{{\rm w}}

\renewcommand{\i}{{\mathbf{i}}}

\newcommand{\PSL}{{\rm PSL}}
\newcommand{\Cinf}{C^{\infty}}

\newcommand{\RP}{{\mathbb{R}\mathrm{P}}}

\def\NABLA#1{{\mathop{\nabla\kern-.5ex\lower1ex\hbox{$#1$}}}}
\def\Nabla#1{\nabla\kern-.5ex{}_{#1}}

\def\Abs#1{\left|#1\right|}

\renewcommand{\p}{{\partial}}

\newtheorem{theorem}{Theorem}[section]

\newtheorem{lemma}[theorem]{Lemma}
\newtheorem{proposition}[theorem]{Proposition}
\newtheorem{definition}[theorem]{Definition}
\newtheorem{remark}[theorem]{Remark}
\newtheorem{example}[theorem]{Example}

\begin{document}

\maketitle

\begin{abstract}
We prove a formula that expresses the Viterbo--Maslov index 
of a smooth strip in an oriented $2$-manifold with boundary 
curves contained in $1$-dimensional submanifolds 
in terms the degree function on the complement of the
union of the two submanifolds.
\end{abstract}

%\tableofcontents

%%%%%%%%%%%%%%%%%%%%%%%%%%%%%%%%%%%%%%%%%%%%%%%%
%%%%%%%%%%%%%%%%%%%%%%%%%%%%%%%%%%%%%%%%%%%%%%%%
%%%%%%%%%%%%%%%%%% Section 1 %%%%%%%%%%%%%%%%%%%
%%%%%%%%%%%%%%%%%%%%%%%%%%%%%%%%%%%%%%%%%%%%%%%%
%%%%%%%%%%%%%%%%%%%%%%%%%%%%%%%%%%%%%%%%%%%%%%%%

\section{Introduction}\label{sec:intro}

We assume throughout this paper that $\Sigma$ 
is a connected oriented 2-manifold without boundary
and $\alpha,\beta\subset\Sigma$ are connected smooth 
one dimensional oriented submanifolds without boundary 
which are closed as subsets of $\Sigma$ and intersect transversally.
We do not assume that $\Sigma$ is compact, but when it is, 
$\alpha$ and $\beta$ are embedded circles. 
Denote the standard half disc by
$$
\D:=\{z\in\C\,|\,\IM\,z\ge 0,\,|z|\le 1\}.
$$
Let $\cD$ denote the space of all smooth maps $u:\D\to\Sigma$ 
satisfying the boundary conditions
$u(\D\cap\R)\subset\alpha$ and $u(\D\cap S^1)\subset\beta$.
For $x,y\in\alpha\cap\beta$ let $\cD(x,y)$
denote the subset of all $u\in\cD$ satisfying the endpoint 
conditions $u(-1)=x$ and $u(1)=y$.
Each $u\in\cD$ determines a locally constant function
${\w:\Sigma\setminus(\alpha\cup\beta)\to\Z}$
defined as the degree
$$
\w(z) := \deg(u,z), \qquad z\in\Sigma\setminus(\alpha\cup\beta).
$$
When $z$ is a regular value of $u$ this is the algebraic 
number of points in the preimage $u^{-1}(z)$. 
The function $\w$ depends only on the homotopy class of~$u$.
We prove that the homotopy class of $u$ is uniquely determined 
by its endpoints $x,y$ and its degree function $\w$
(Theorem~\ref{thm:trace}).  
The main theorem of this paper asserts that the 
Viterbo--Maslov index of an element ${u\in\cD(x,y)}$ 
is given by the formula
\begin{equation}\label{eq:MV}
\mu(u) = \frac{m_x+m_y}{2},
\end{equation}
where $m_x$ denotes the sum of the four values 
of $\w$ encountered when walking along a 
small circle surrounding $x$, and similarly for $y$
(Theorem~\ref{thm:maslov}). The formula~\eqref{eq:MV}
plays a central role in our combinatorial approach~\cite{DESILVA,RSS}
to Floer homology~\cite{FLOER1,FLOER2}.
An appendix contains a proof that the space of paths 
connecting $\alpha$ to $\beta$ is simply 
connected under suitable assumptions.

\medskip\noindent{\bf Acknowledgement.}
We thank David Epstein for explaining 
to us the proof of Proposition~\ref{prop:dbae}.

%%%%%%%%%%%%%%%%%%%%%%%%%%%%%%%%%%%%%%%%%%%%%%%%
%%%%%%%%%%%%%%%%%%%%%%%%%%%%%%%%%%%%%%%%%%%%%%%%
%%%%%%%%%%%%%%%%%% Section 2 %%%%%%%%%%%%%%%%%%%
%%%%%%%%%%%%%%%%%%%%%%%%%%%%%%%%%%%%%%%%%%%%%%%%
%%%%%%%%%%%%%%%%%%%%%%%%%%%%%%%%%%%%%%%%%%%%%%%%

\section{Chains and Traces}\label{sec:CT}

Define a cell complex structure on $\Sigma$ by taking  
the set of zero-cells to be the set $\alpha\cap\beta$,
the set of one-cells to be  the set of connected components 
of $(\alpha\setminus\beta)\cup(\beta\setminus\alpha)$
with compact closure, and the set of two-cells to be the 
set of connected components of $\Sigma\setminus(\alpha\cup\beta)$
with compact closure.  (There is an abuse of language here 
as the ``two-cells'' need not be  homeomorphs of the open unit
disc if the genus of $\Sigma$ is positive and the ``one-cells" 
need not be arcs if $\alpha\cap\beta=\emptyset$.)
Define a boundary operator $\p$ as follows. 
For each two-cell $F$ let 
$
\p F=\sum \pm E, 
$
where the sum is over the one-cells $E$ which abut $F$ and 
the plus sign is chosen iff the orientation of $E$ (determined 
from the given orientations of $\alpha$ and $\beta$)
agrees with the boundary orientation of $F$
as a connected open subset of the oriented manifold  $\Sigma$.
For each one-cell $E$ let 
$
\p E=b-a
$
where $a$ and $b$ are the endpoints of the arc $E$ and the 
orientation of $E$ goes from $a$ to $b$. (The one-cell $E$ 
is either a subarc of $\alpha$ or a subarc of $\beta$ and both 
$\alpha$ and $\beta$ are oriented one-manifolds.) 
For $k=0,1,2$ a {\bf $k$-chain} is defined to be a formal linear
combination (with integer coefficients) of $k$-cells, i.e.\ 
a two-chain is a locally constant map 
$\Sigma\setminus(\alpha\cup\beta)\to\Z$ 
(whose support has compact closure in $\Sigma$)
and a one-chain is a locally constant map 
$(\alpha\setminus\beta)\cup(\beta\setminus\alpha)\to\Z$
(whose support has compact closure in $\alpha\cup\beta$).
It follows directly from the definitions that
$\p^2F=0$  for each two-cell $F$.

Each $u\in\cD$ determines a two-chain $\w$ via
\begin{equation}\label{eq:w}
\w(z) := \deg(u,z), \qquad z\in\Sigma\setminus(\alpha\cup\beta).
\end{equation}
and a one-chain $\nu$ via
\begin{equation}\label{eq:nu}
\nu(z) := 
\left\{\begin{array}{rl}
\deg(u\big|_{\p\D\cap\R\;\,}:\p\D\cap\R\;\to\alpha,z),&
\mbox{for } z\in\alpha\setminus\beta, \\
-\deg(u\big|_{\p\D\cap S^1}:\p\D\cap S^1\to\beta,z),&
\mbox{for } z\in\beta\setminus\alpha.
\end{array}\right.
\end{equation}
Here we orient the one-manifolds $\D\cap\R$ and 
$\D\cap S^1$ from $-1$ to $+1$.  For any one-chain 
$\nu:(\alpha\setminus\beta)\cup(\beta\setminus\alpha)\to\Z$
denote 
$$
\nu_\alpha:=\nu|_{\alpha\setminus\beta}:
\alpha\setminus\beta\to\Z,\qquad
\nu_\beta:=\nu|_{\alpha\setminus\beta}:
\beta\setminus\alpha\to\Z.
$$
Conversely, given locally constant functions 
$\nu_\alpha:\alpha\setminus\beta\to\Z$ and
$\nu_\beta:\beta\setminus\alpha\to\Z$,
denote by $\nu=\nu_\alpha-\nu_\beta$ the one-chain that 
agrees with $\nu_\alpha$ on $\alpha\setminus\beta$ and 
agrees with $-\nu_\beta$ on $\beta\setminus\alpha$.

\begin{definition}[{\bf Traces}]\label{def:trace}\rm
Fix two (not necessarily distinct)
intersection points $x,y\in\alpha\cap\beta$.

\smallskip\noindent{\bf (i)}
Let $\w:\Sigma\setminus(\alpha\cup\beta)\to\Z$
be a two-chain.  The triple 
$
\Lambda=(x,y,\w)
$
is called an {\bf $(\alpha,\beta)$-trace} if 
there exists an element $u\in\cD(x,y)$ such that
$\w$ is given by~\eqref{eq:w}. 
In this case $\Lambda=:\Lambda_u$ 
is also called the {\bf $(\alpha,\beta)$-trace of $u$}
and we sometimes write $\w_u:=\w$.

\smallskip\noindent{\bf (ii)}
Let $\Lambda=(x,y,\w)$ be an $(\alpha,\beta)$-trace.
The triple $\p\Lambda:=(x,y,\p\w)$
is called the {\bf boundary of $\Lambda$.}

\smallskip\noindent{\bf (iii)}
A one-chain 
$\nu:(\alpha\setminus\beta)\cup(\beta\setminus\alpha)\to\Z$
is called an {\bf $(x,y)$-trace} if there exist 
smooth curves $\gamma_\alpha:[0,1]\to\alpha$
and $\gamma_\beta:[0,1]\to\beta$ such that
$\gamma_\alpha(0)=\gamma_\beta(0)=x$,
$\gamma_\alpha(1)=\gamma_\beta(1)=y$,
$\gamma_\alpha$ and $\gamma_\beta$ 
are homotopic in $\Sigma$ with fixed endpoints, and 
\begin{equation}\label{eq:admissible}
\nu(z)=\left\{\begin{array}{rl}
\deg(\gamma_\alpha,z),&\mbox{for }z\in\alpha\setminus\beta,\\
-\deg(\gamma_\beta,z),&\mbox{for }z\in\beta\setminus\alpha.
\end{array}\right.
\end{equation}
\end{definition}

\begin{remark}\label{rmk:nu}\rm
Assume $\Sigma$ is simply connected.
Then the condition on $\gamma_\alpha$ and $\gamma_\beta$ 
to be homotopic with fixed endpoints is redundant.
Moreover, if $x=y$ then a one-chain $\nu$ 
is an $(x,y)$-trace if and only if the restrictions
$\nu_\alpha:=\nu|_{\alpha\setminus\beta}$ and 
$\nu_\beta:=-\nu|_{\beta\setminus\alpha}$ are constant.
If $x\ne y$ and $\alpha,\beta$ are embedded circles 
and $A,B$ denote the positively oriented arcs from $x$ to $y$ 
in $\alpha,\beta$, then a one-chain $\nu$ 
is an $(x,y)$-trace if and only if 
$
\nu_\alpha|_{\alpha\setminus(A\cup\beta)}=\nu_\alpha|_{A\setminus\beta}-1
$
and
$
\nu_\beta|_{\beta\setminus(B\cup\alpha)}=\nu_\beta|_{B\setminus\alpha}-1.
$
In particular, when walking along $\alpha$ or $\beta$, 
the function $\nu$ only changes its value at $x$ and $y$.
\end{remark}

\begin{lemma}\label{le:boundary}
Let $x,y\in\alpha\cap\beta$ and $u\in\cD(x,y)$.
Then the boundary of the $(\alpha,\beta)$-trace
$\Lambda_u$ of $u$ is the triple 
$
\p\Lambda_u=(x,y,\nu),
$
where $\nu$ is given by~\eqref{eq:nu}.
In other words, if $\w$ is given by~\eqref{eq:w}
and $\nu$ is given by~\eqref{eq:nu}
then $\nu=\p\w$.
\end{lemma}

\begin{proof}
Choose an embedding $\gamma:[-1,1]\to\Sigma$ 
such that $u$ is transverse to $\gamma$,  
$\gamma(t)\in\Sigma\setminus(\alpha\cup\beta)$ for $t\ne 0$, 
$\gamma(-1)$, $\gamma(1)$ are regular values of $u$, 
${\gamma(0)\in\alpha\setminus\beta}$ 
is a regular value of $u|_{\D\cap\R}$,
and $\gamma$ intersects $\alpha$ transversally 
at $t=0$ such that orientations match in
$$
T_{\gamma(0)}\Sigma = T_{\gamma(0)}\alpha \oplus\R\dot\gamma(0).
$$
Denote $\Gamma:=\gamma([-1,1])$. Then $u^{-1}(\Gamma)\subset\D$ 
is a $1$-dimensional submanifold with boundary
$$
\p u^{-1}(\Gamma) = u^{-1}(\gamma(-1))\cup u^{-1}(\gamma(1))
\cup \bigl(u^{-1}(\gamma(0))\cap\R)\bigr).
$$
If $z\in u^{-1}(\Gamma)$ then 
$$
\im\,du(z)+T_{u(z)}\Gamma=T_{u(z)}\Sigma,\qquad
T_zu^{-1}(\Gamma)=du(z)^{-1}T_{u(z)}\Gamma.
$$ 
We orient $u^{-1}(\Gamma)$ such that the orientations 
match in
$$
T_{u(z)}\Sigma = T_{u(z)}\Gamma\oplus du(z)\i T_zu^{-1}(\Gamma).
$$
In other words, if $z\in u^{-1}(\Gamma)$ and $u(z)=\gamma(t)$, 
then a nonzero tangent vector $\zeta\in T_zu^{-1}(\Gamma)$
is positive if and only if the pair $(\dot\gamma(t),du(z)\i\zeta)$ 
is a positive basis of $T_{\gamma(t)}\Sigma$.
Then the boundary orientation of $u^{-1}(\Gamma)$ 
at the elements of $u^{-1}(\gamma(1))$ agrees with the 
algebraic count  in the definition of $\w(\gamma(1))$,  
at the elements of $u^{-1}(\gamma(-1))$ is opposite 
to the algebraic count in the definition of $\w(\gamma(-1))$,
and at the elements of $u^{-1}(\gamma(0))\cap\R$ is opposite 
to the algebraic count in the definition of $\nu(\gamma(0))$.
Hence
$$
\w(\gamma(1)) = \w(\gamma(-1)) + \nu(\gamma(0)).
$$
In other words the value of $\nu$ at a point in
$\alpha\setminus\beta$ is equal to the value 
of~$\w$ slightly to the left of $\alpha$ 
minus the value of~$\w$ slightly to the right of~$\alpha$.
Likewise, the value of $\nu$ at a point in
$\beta\setminus\alpha$ is equal to the value 
of~$\w$ slightly to the right of $\beta$ 
minus the value of~$\w$ slightly to the left of~$\beta$.
This proves Lemma~\ref{le:boundary}.
\end{proof}

\begin{theorem}\label{thm:trace}
\smallskip\noindent{\bf (i)}
Two elements of $\cD$ belong to the same connected
component of~$\cD$ if and only if they have 
the same $(\alpha,\beta)$-trace.

\smallskip\noindent{\bf (ii)}
Assume $\Sigma$ is diffeomorphic to the two-sphere.
Then $\Lambda=(x,y,\w)$
is an $(\alpha,\beta)$-trace if and only if 
$\p\w$ is an $(x,y)$-trace.

\smallskip\noindent{\bf (iii)}
Assume $\Sigma$ is not diffeomorphic to the two-sphere
and let $x,y\in\alpha\cap\beta$. If $\nu$ is an $(x,y)$-trace,
then there is a unique two-chain $\w$ such that 
$\Lambda:=(x,y,\w)$ is an $(\alpha,\beta)$-trace
and $\p\w=\nu$.
\end{theorem}

\begin{proof} 
We prove~(i). ``Only if'' follows from the standard 
arguments in degree theory as in Milnor~\cite{M}.  
To prove ``if'', fix two intersection points 
$$
x,y\in\alpha\cap\beta
$$ 
and, for $X=\Sigma,\alpha,\beta$, 
denote by $\cP(x,y;X)$ the space of all smooth curves 
${\gamma:[0,1]\to X}$ satisfying $\gamma(0)=x$ and $\gamma(1)=y$.
Every $u\in\cD(x,y)$ determines smooth paths 
$
\gamma_{u,\alpha}\in\cP(x,y;\alpha)
$
and
$
\gamma_{u,\beta}\in\cP(x,y;\beta)
$ 
via
\begin{equation}\label{eq:gaualbe}
\gamma_{u,\alpha}(s):=u(-\cos(\pi s),0),\qquad
\gamma_{u,\beta}(s)=u(-\cos(\pi s),\sin(\pi s)).
\end{equation}
These paths are homotopic in $\Sigma$ with fixed
endpoints. An explicit homotopy is the map 
$$
F_u:=u\circ\phi:[0,1]^2\to\Sigma
$$
where $\phi:[0,1]^2\to\D$ is the map
$$
\phi(s,t):=(-\cos(\pi s),t\sin(\pi s)).
$$
By Lemma~\ref{le:boundary}, he homotopy class of 
$\gamma_{u,\alpha}$ in $\cP(x,y;\alpha)$ is uniquely determined by 
$\nu_\alpha:=\p\w_u|_{\alpha\setminus\beta}:\alpha\setminus\beta\to\Z$
and that of $\gamma_{u,\beta}$ in $\cP(x,y;\beta)$ is uniquely determined 
by $\nu_\beta:=-\p\w_u|_{\beta\setminus\alpha}:\beta\setminus\alpha\to\Z$.  
Hence they are both uniquely determined by the $(\alpha,\beta)$-trace
of $u$. If $\Sigma$ is not diffeomorphic to the $2$-sphere 
the assertion follows from the fact that
each component of $\cP(x,y;\Sigma)$ is contractible
(because the universal cover of $\Sigma$ is diffeomorphic
to the complex plane).  Now assume $\Sigma$ is diffeomorphic 
to the $2$-sphere. Then $\pi_1(\cP(x,y;\Sigma))=\Z$
acts on $\pi_0(\cD)$ because the correspondence $u\mapsto F_u$
identifies $\pi_0(\cD)$ with a space of homotopy classes of 
paths in $\cP(x,y;\Sigma)$ connecting 
$\cP(x,y;\alpha)$ to $\cP(x,y;\beta)$.   
The induced action on the space of two-chains 
$\w:\Sigma\setminus(\alpha\cup\beta)$ is given by adding 
a global constant.  Hence the map $u\mapsto\w$ induces an injective 
map 
$$
\pi_0(\cD(x,y))\to\{\mbox{$2$-chains}\}.
$$
This proves~(i).

\smallbreak

We prove~(ii) and~(iii). 
Let $\w$ be a two-chain, suppose that 
$$
\nu:=\p\w
$$
is an $(x,y)$-trace, and denote
$$
\Lambda := (x,y,\w).
$$
Let ${\gamma_\alpha:[0,1]\to\alpha}$ and 
${\gamma_\beta:[0,1]\to\beta}$ 
be as in Definition~\ref{def:trace}.
Then there is a $u'\in\cD(x,y)$ such that the map
$s\mapsto u'(-\cos(\pi s),0)$ is homotopic to $\gamma_\alpha$
and $s\mapsto u'(-\cos(\pi s),\sin(\pi s))$ is homotopic to 
$\gamma_\beta$.  By definition the $(\alpha,\beta)$-trace of $u'$
is $\Lambda'=(x,y,\w')$ for some two-chain $\w'$.  
By Lemma~\ref{le:boundary}, we have
$$
\p\w'=\nu=\p\w
$$
and hence $\w-\w'=:d$ is constant.
If $\Sigma$ is not diffeomorphic to the two-sphere
and $\Lambda$ is the $(\alpha,\beta)$-trace of some element 
$u\in\cD$, then $u$ is homotopic to $u'$ (as $\cP(x,y;\Sigma)$ 
is simply connected) and hence $d=0$ and $\Lambda=\Lambda'$.
If $\Sigma$ is diffeomorphic to the $2$-sphere choose
a smooth map $v:S^2\to \Sigma$ of degree $d$ and replace $u'$ 
by the connected sum $u:=u'\# v$.  
Then $\Lambda$ is the $(\alpha,\beta)$-trace of $u$.
This proves Theorem~\ref{thm:trace}.
\end{proof}

\begin{remark}\label{rmk:trace}\rm
Let $\Lambda=(x,y,\w)$ be an 
$(\alpha,\beta)$-trace
and define 
$$
\nu_\alpha:=\p\w|_{\alpha\setminus\beta},\qquad
\nu_\beta:=-\p\w|_{\beta\setminus\alpha}.
$$

\smallskip\noindent{\bf (i)}
The two-chain $\w$ is uniquely 
determined by the condition $\p\w=\nu_\alpha-\nu_\beta$ 
and its value at one point.  To see this, think of the 
embedded circles $\alpha$ and $\beta$ as traintracks. 
Crossing $\alpha$ at a point $z\in\alpha\setminus\beta$
increases $\w$ by $\nu_\alpha(z)$
if the train comes from the left, and decreases it 
by~$\nu_\alpha(z)$ if the train comes from the right.
Crossing $\beta$ at a point $z\in\beta\setminus\alpha$
decreases $\w$ by $\nu_\beta(z)$
if the train comes from the left 
and increases it by $\nu_\beta(z)$
if the train comes from the right.
Moreover, $\nu_\alpha$ extends continuously to $\alpha\setminus\{x,y\}$
and $\nu_\beta$ extends continuously to $\beta\setminus\{x,y\}$.
At each intersection point $z\in(\alpha\cap\beta)\setminus\{x,y\}$ 
with intersection index $+1$ (respectively $-1$)
the function $\w$ takes the values
$$
k,\quad
k+\nu_\alpha(z),\quad
k+\nu_\alpha(z)-\nu_\beta(z),\quad
k-\nu_\beta(z)
$$ 
as we march counterclockwise (respectively clockwise)
along a small circle surrounding the intersection point.

\smallskip\noindent{\bf (ii)}
If $\Sigma$ is not diffeomorphic to the $2$-sphere
then, by Theorem~\ref{thm:trace}~(iii), 
the $(\alpha,\beta)$-trace $\Lambda$ is uniquely determined 
by its boundary $\p\Lambda=(x,y,\nu_\alpha-\nu_\beta)$.

\smallskip\noindent{\bf (iii)}
Assume $\Sigma$ is not diffeomorphic to the $2$-sphere
and choose a universal covering $\pi:\C\to\Sigma$.
Choose a point $\tx\in\pi^{-1}(x)$ and lifts 
$\tal$ and $\tbe$ of~$\alpha$ and~$\beta$ 
such that
$
\tx\in\tal\cap\tbe.
$
Then $\Lambda$ lifts to an $(\tal,\tbe)$-trace  
$$
\tLa = (\tx,\ty,\tilde\w).
$$
More precisely, the one chain $\nu:=\nu_\alpha-\nu_\beta=\p\w$ 
is an $(x,y)$-trace, by Lemma~\ref{le:boundary}.  The paths 
$\gamma_\alpha:[0,1]\to\alpha$ and 
$\gamma_\beta:[0,1]\to\beta$ in Definition~\ref{def:trace} 
lift to unique paths $\gamma_\tal:[0,1]\to\tal$ and 
$\gamma_\tbe:[0,1]\to\tbe$ connecting $\tx$ to $\ty$.
For $\tz\in\C\setminus(\tA\cup\tB)$ the number
$\tilde\w(\tz)$ is the winding number of the loop 
$
\gamma_\tal-\gamma_\tbe
$ 
about $\tz$ (by Rouch\'e's theorem). 
The two-chain $\w$ is then given by 
$$
\w(z)
= \sum_{\tz\in\pi^{-1}(z)} \tilde\w(\tz),\qquad
z\in\Sigma\setminus(\alpha\cup\beta).
$$
To see this, lift an element $u\in\cD(x,y)$ with 
$(\alpha,\beta)$-trace $\Lambda$ to the 
universal cover to obtain an element 
$\tu\in\cD(\tx,\ty)$ with $\Lambda_\tu=\tLa$
and consider the degree.
\end{remark}

\begin{definition}[{\bf Catenation}]\label{def:catenation}\rm
Let $x,y,z\in\alpha\cap\beta$. 
The {\bf catenation of two $(\alpha,\beta)$-traces
$\Lambda=(x,y,\w)$ and $\Lambda'=(y,z,\w')$} is defined by 
$$
\Lambda\#\Lambda' := (x,z,\w+\w').
$$
Let $u\in\cD(x,y)$ and $u'\in\cD(y,z)$
and suppose that $u$ and $u'$ are constant 
near the ends $\pm1\in\D$. 
For $0<\lambda<1$ sufficiently close to one 
the {\bf $\lambda$-catenation of $u$ and $u'$} 
is the map $u\#_\lambda u'\in\cD(x,z)$ defined by
$$
(u\#_\lambda u')(\zeta) := \left\{\begin{array}{ll}
u\left(\frac{\zeta+\lambda}{1+\lambda\zeta}\right),&
\mbox{for }\Re\,\zeta\le 0,\\
u'\left(\frac{\zeta-\lambda}{1-\lambda\zeta}\right),&
\mbox{for }\Re\,\zeta\ge 0.
\end{array}\right.
$$
\end{definition}

\begin{lemma}\label{le:catenation}
If $u\in\cD(x,y)$ and $u'\in\cD(y,z)$ are as in 
Definition~\ref{def:catenation} then
$$
\Lambda_{u\#_\lambda u'} = \Lambda_u \#\Lambda_{u'}.
$$
Thus the catenation of two $(\alpha,\beta)$-traces 
is again an $(\alpha,\beta)$-trace.
\end{lemma}

\begin{proof}
This follows directly from the definitions.
\end{proof}

%%%%%%%%%%%%%%%%%%%%%%%%%%%%%%%%%%%%%%%%%%%%%%%%
%%%%%%%%%%%%%%%%%%%%%%%%%%%%%%%%%%%%%%%%%%%%%%%%
%%%%%%%%%%%%%%%%%% Section 3 %%%%%%%%%%%%%%%%%%%
%%%%%%%%%%%%%%%%%%%%%%%%%%%%%%%%%%%%%%%%%%%%%%%%
%%%%%%%%%%%%%%%%%%%%%%%%%%%%%%%%%%%%%%%%%%%%%%%%

\section{The Maslov Index}\label{sec:M}

\begin{definition}\label{def:maslov}\rm
Let $x,y\in\alpha\cap\beta$ and $u\in\cD(x,y)$.
Choose an orientation preserving trivialization 
$$
\D\times\R^2\to u^*T\Sigma:(z,\zeta)\mapsto\Phi(z)\zeta
$$
and consider the Lagrangian paths 
$$
\lambda_0,\lambda_1:[0,1]\to\RP^1
$$
given by 
\begin{equation*}
\begin{split}
\lambda_0(s)&:=\Phi(-\cos(\pi s),0)^{-1}
T_{u(-\cos(\pi s),0)}\alpha,\\
\lambda_1(s)&:=\Phi(-\cos(\pi s),\sin(\pi s))^{-1}
T_{u(-\cos(\pi s),\sin(\pi s))}\beta.
\end{split}
\end{equation*}
The {\bf Viterbo--Maslov index of $u$} is defined as the 
relative Maslov index of the pair of Lagrangian paths
$(\lambda_0,\lambda_1)$ and will be denoted by
$$
\mu(u) := \mu(\Lambda_u) := \mu(\lambda_0,\lambda_1).
$$
By the naturality and homotopy axioms for the relative Maslov index 
(see for example~\cite{RS2}), the number $\mu(u)$ 
is independent of the choice of the trivialization 
and depends only on the homotopy class of $u$;
hence it depends only on the $(\alpha,\beta)$-trace of $u$,
by Theorem~\ref{thm:trace}.   The relative Maslov index 
$\mu(\lambda_0,\lambda_1)$ is the degree of the loop in 
$\RP^1$ obtained by traversing $\lambda_0$,  followed by 
a counterclockwise turn from $\lambda_0(1)$ to $\lambda_1(1)$,
followed by traversing $\lambda_1$ in reverse time, followed
by a clockwise turn from $\lambda_1(0)$ to $\lambda_0(0)$.
This index was first defined by Viterbo~\cite{VITERBO} 
(in all dimensions).  Another exposition is contained in~\cite{RS2}.
\end{definition}

\begin{remark}\label{rmk:mascat}\rm
The Viterbo--Maslov index is additive under catenation,
i.e.\ if 
$$
\Lambda=(x,y,\w),\qquad
\Lambda'=(y,z,\w')
$$ 
are $(\alpha,\beta)$-traces then
$$
\mu(\Lambda\#\Lambda') = \mu(\Lambda)+\mu(\Lambda').
$$
For a proof of this formula see~\cite{VITERBO,RS2}.
\end{remark}

\begin{definition}\label{def:arc}\rm
Let $\Lambda=(x,y,\w)$ be an $(\alpha,\beta)$-trace
and 
$$
\nu_\alpha:=\p\w|_{\alpha\setminus\beta},\qquad
\nu_\beta:=-\p\w|_{\beta\setminus\alpha}.
$$
$\Lambda$ is said to satisfy the {\bf arc condition} if
\begin{equation}\label{eq:arc}
x\ne y,\qquad \min\Abs{\nu_\alpha} = \min\Abs{\nu_\beta}=0.
\end{equation}
When $\Lambda$ satisfies the arc condition
there are arcs $A\subset\alpha$ and $B\subset\beta$ 
from $x$ to $y$ such that
\begin{equation}\label{eq:nuAB}
\nu_\alpha(z) = \left\{\begin{array}{rl}
\pm1,&\mbox{if }z\in A,\\
0,&\mbox{if }z\in\alpha\setminus\overline A,
\end{array}\right.\;\;
\nu_\beta(z) = \left\{\begin{array}{rl}
\pm1,&\mbox{if }z\in B,\\
0,&\mbox{if }z\in\beta\setminus\overline B.
\end{array}\right.
\end{equation}
Here the plus sign is chosen iff the orientation of $A$ from 
$x$ to $y$ agrees with that of $\alpha$, respectively the 
orientation of $B$ from $x$ to $y$ agrees with that of~$\beta$.
In this situation the quadruple $(x,y,A,B)$ and the triple
$(x,y,\p\w)$ determine one another and we also write
$$
\p\Lambda = (x,y,A,B)
$$ 
for the boundary of $\Lambda$.
When $u\in\cD$ and $\Lambda_u=(x,y,\w)$ 
satisfies the arc condition and $\p\Lambda_u=(x,y,A,B)$
then 
$$
s\mapsto u(-\cos(\pi s),0)
$$
is homotopic in $\alpha$ 
to a path traversing $A$ and the path 
$$
s\mapsto u(-\cos(\pi s),\sin(\pi s))
$$
is homotopic in~$\beta$ to a path traversing $B$.
\end{definition}

\begin{theorem}\label{thm:maslov}
Let $\Lambda=(x,y,\w)$
be an $(\alpha,\beta)$-trace.  For $z\in\alpha\cap\beta$
denote by $m_z(\Lambda)$ the sum of the four 
values of $\w$ encountered when walking along 
a small circle surrounding~$z$.
Then the Viterbo--Maslov index of~$\Lambda$ is given by 
\begin{equation}\label{eq:maslov}
\mu(\Lambda) = \frac{m_x(\Lambda)+m_y(\Lambda)}{2}.
\end{equation}
\end{theorem}

We first prove the result for the $2$-plane and the $2$-sphere
(Section~\ref{sec:PLANE}).  When $\Sigma$ is not 
simply connected we reduce the result to the case 
of the $2$-plane (Section~\ref{sec:LIFT}).  
The key is the identity
\begin{equation}\label{eq:mtilde}
m_{g\tx}(\tLa) + m_{g^{-1}\ty}(\tLa)=0
\end{equation}
for every lift $\tLa$ to the universal cover and every deck
transformation $g\ne\id$.  

%%%%%%%%%%%%%%%%%%%%%%%%%%%%%%%%%%%%%%%%%%%%%%%%
%%%%%%%%%%%%%%%%%%%%%%%%%%%%%%%%%%%%%%%%%%%%%%%%
%%%%%%%%%%%%%%%%%% Section 4 %%%%%%%%%%%%%%%%%%%
%%%%%%%%%%%%%%%%%%%%%%%%%%%%%%%%%%%%%%%%%%%%%%%%
%%%%%%%%%%%%%%%%%%%%%%%%%%%%%%%%%%%%%%%%%%%%%%%%

\section{The Simply Connected Case}\label{sec:PLANE}

A connected oriented $2$-manifold $\Sigma$
is called {\bf planar} if it admits an (orientation preserving)
embedding into the complex plane. 

\begin{proposition}\label{prop:maslovC}
Equation~\eqref{eq:maslov} holds when $\Sigma$ is planar.
\end{proposition}

\begin{proof}
Assume first that $\Sigma=\C$ and 
$
\Lambda=(x,y,\w)
$ 
satisfies the arc condition.  Thus
the boundary of $\Lambda$ has the form
$$
\p\Lambda = (x,y,A,B),
$$
where $A\subset\alpha$ and $B\subset\beta$ are arcs from 
$x$ to $y$ and $\w(z)$ is the winding number of the loop $A-B$ 
about the point $z\in\Sigma\setminus(A\cup B)$
(see Remark~\ref{rmk:trace}). Hence the formula~\eqref{eq:maslov} 
can be written in the form
\begin{equation}\label{eq:masLov}
\mu(\Lambda) = 2k_x + 2k_y + \frac{\eps_x-\eps_y}{2}.
\end{equation}
Here 
$
\eps_z=\eps_z(\Lambda)\in\{+1,-1\}
$ 
denotes the intersection index of $A$ and $B$ 
at a point $z\in A\cap B$, $k_x=k_x(\Lambda)$ 
denotes the value of the winding number $\w$ 
at a point in $\alpha\setminus A$ close to $x$,
and $k_y=k_y(\Lambda)$ denotes the value of $\w$ 
at a point in $\alpha\setminus A$ close to $y$.
We now prove~\eqref{eq:masLov} under the assumption
that $\Lambda$ satisfies the arc condition.
The proof is by induction on the number of intersection 
points of $B$ and~$\alpha$ and has seven steps.

\medskip\noindent{\bf Step~1.}
{\it We may assume without loss of generality that
\begin{equation}\label{eq:AB}
\Sigma=\C,\qquad \alpha=\R,\qquad A=[x,y],\qquad x<y,
\end{equation}
and $B\subset\C$ is an embedded arc from $x$ to $y$ 
that is transverse to $\R$.}

\medskip\noindent
Choose a diffeomorphism from $\Sigma$ 
to $\C$ that maps $A$ to a bounded closed interval 
and maps $x$ to the left endpoint of $A$.
If $\alpha$ is not compact the diffeomorphism can be chosen
such that it also maps $\alpha$ to $\R$.
If $\alpha$ is an embedded circle the diffeomorphism 
can be chosen such that its restriction to $B$ is 
transverse to $\R$; now replace the image of~$\alpha$ by~$\R$. 
This proves Step~1.

\medskip\noindent{\bf Step~2.}
{\it Assume~\eqref{eq:AB} and let 
$\bar\Lambda:=(x,y,z\mapsto -\w(\bar z))$ 
be the $(\alpha,\bar\beta)$-trace 
obtained from $\Lambda$ by complex conjugation.
Then $\Lambda$ satisfies~\eqref{eq:masLov} if and only if 
$\bar\Lambda$ satisfies~\eqref{eq:masLov}.}

\medskip\noindent
Step~2 follows from the fact that the numbers
$\mu,k_x,k_y,\eps_x,\eps_y$ change sign 
under complex conjugation.

\medskip\noindent{\bf Step~3.}
{\it  Assume~\eqref{eq:AB}.  If $B\cap\R=\{x,y\}$ then $\Lambda$ 
satisfies~\eqref{eq:masLov}.}

\medskip\noindent
In this case $B$ is contained in the upper 
or lower closed half plane and the loop $A\cup B$ 
bounds a disc contained in the same half plane.
By Step~1 we may assume that $B$ is contained in the 
upper half space. Then $\eps_x=1$, $\eps_y=-1$,
and $\mu(\Lambda)=1$.  Moreover, the winding number 
$\w$ is one in the disc encircled by $A$ 
and $B$ and is zero in the complement of its closure.  
Since the intervals $(-\infty,0)$ and $(0,\infty)$ 
are contained in this complement, we have $k_x=k_y=0$.
This proves Step~3.

\medskip\noindent{\bf Step~4.}
{\it Assume~\eqref{eq:AB} and $\#(B\cap\R)>2$, 
follow the arc of $B$, starting at $x$, 
and let~$x'$ be the next intersection point with $\R$. 
Assume $x'<x$, denote by $B'$ the arc in $B$ from $x'$ to $y$, 
and let $A':=[x',y]$ (see Figure~\ref{fig:maslov4}).
If the $(\alpha,\beta)$-trace 
$\Lambda'$ with boundary $\p\Lambda'=(x',y,A',B')$ 
satisfies~\eqref{eq:masLov} so does $\Lambda$.}

\begin{figure}[htp] 
\centering 
\includegraphics[scale=0.4]{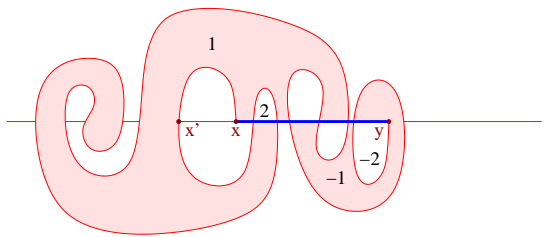}      
\caption{{Maslov index and catenation: $x'<x<y$.}}\label{fig:maslov4}      
\end{figure}

\medskip\noindent
By Step~2 we may assume $\eps_x(\Lambda)=1$.
Orient $B$ from $x$ to $y$.
The Viterbo--Maslov index of $\Lambda$ is 
minus the Maslov index of the path 
$
B\to\RP^1:z\mapsto T_zB, 
$
relative to the Lagrangian subspace $\R\subset\C$.
Since the Maslov index of the arc in $B$ from 
$x$ to $x'$ is $+1$ we have
\begin{equation}\label{eq:case11}
\mu(\Lambda) = \mu(\Lambda')-1.
\end{equation}
Since the orientations of $A'$ and $B'$ 
agree with those of $A$ and $B$ we have
\begin{equation}\label{eq:case12}
\eps_{x'}(\Lambda') = \eps_{x'}(\Lambda)=-1,\qquad
\eps_y(\Lambda')=\eps_y(\Lambda).
\end{equation}
Now let $x_1<x_2<\cdots<x_m<x$ be the intersection
points of $\R$ and $B$ in the interval $(-\infty,x)$
and let $\eps_i\in\{-1,+1\}$ be the intersection
index of $\R$ and $B$ at $x_i$. 
Then there is an integer $\ell\in\{1,\dots,m\}$ 
such that $x_\ell=x'$ and $\eps_\ell=-1$. 
Moreover, the winding number $\w$
slightly to the left of $x$ is 
$$
k_x(\Lambda) = \sum_{i=1}^m\eps_i.
$$
It agrees with the value of $\w$ 
slightly to the right of $x'=x_\ell$.  Hence 
\begin{equation}\label{eq:case13}
k_x(\Lambda) 
= \sum_{i=1}^\ell \eps_i 
= \sum_{i=1}^{\ell-1}\eps_i - 1 
= k_{x'}(\Lambda') - 1,\qquad
k_y(\Lambda')=k_y(\Lambda).
\end{equation}
It follows from equation~\eqref{eq:masLov} for $\Lambda'$
and equations~\eqref{eq:case11}, \eqref{eq:case12}, 
and~\eqref{eq:case13} that
\begin{eqnarray*}
\mu(\Lambda) 
&=& 
\mu(\Lambda')-1 \\
&=& 
2k_{x'}(\Lambda')+2k_y(\Lambda')
+ \frac{\eps_{x'}(\Lambda')-\eps_y(\Lambda')}{2}
- 1 \\
&=&
2k_{x'}(\Lambda')+2k_y(\Lambda')
+ \frac{-1-\eps_y(\Lambda)}{2}
- 1 \\
&=&
2k_{x'}(\Lambda')+2k_y(\Lambda')
+ \frac{1-\eps_y(\Lambda)}{2}
- 2  \\
&=&
2k_x(\Lambda)+2k_y(\Lambda)
+ \frac{\eps_x(\Lambda)-\eps_y(\Lambda)}{2}.
\end{eqnarray*}
This proves Step~4.

\medskip\noindent{\bf Step~5.}
{\it Assume~\eqref{eq:AB} and $\#(B\cap\R)>2$, 
follow the arc of $B$, starting at $x$, 
and let~$x'$ be the next intersection 
point with $\R$. Assume $x<x'<y$, 
denote by $B'$ the arc in $B$ from $x'$ to $y$, 
and let $A':=[x',y]$ (see Figure~\ref{fig:maslov5}).  
If the $(\alpha,\beta)$-trace
$\Lambda'$ with boundary $\p\Lambda'=(x',y,A',B')$ 
satisfies~\eqref{eq:masLov} so does $\Lambda$.}
\begin{figure}[htp] 
\centering 
\includegraphics[scale=0.4]{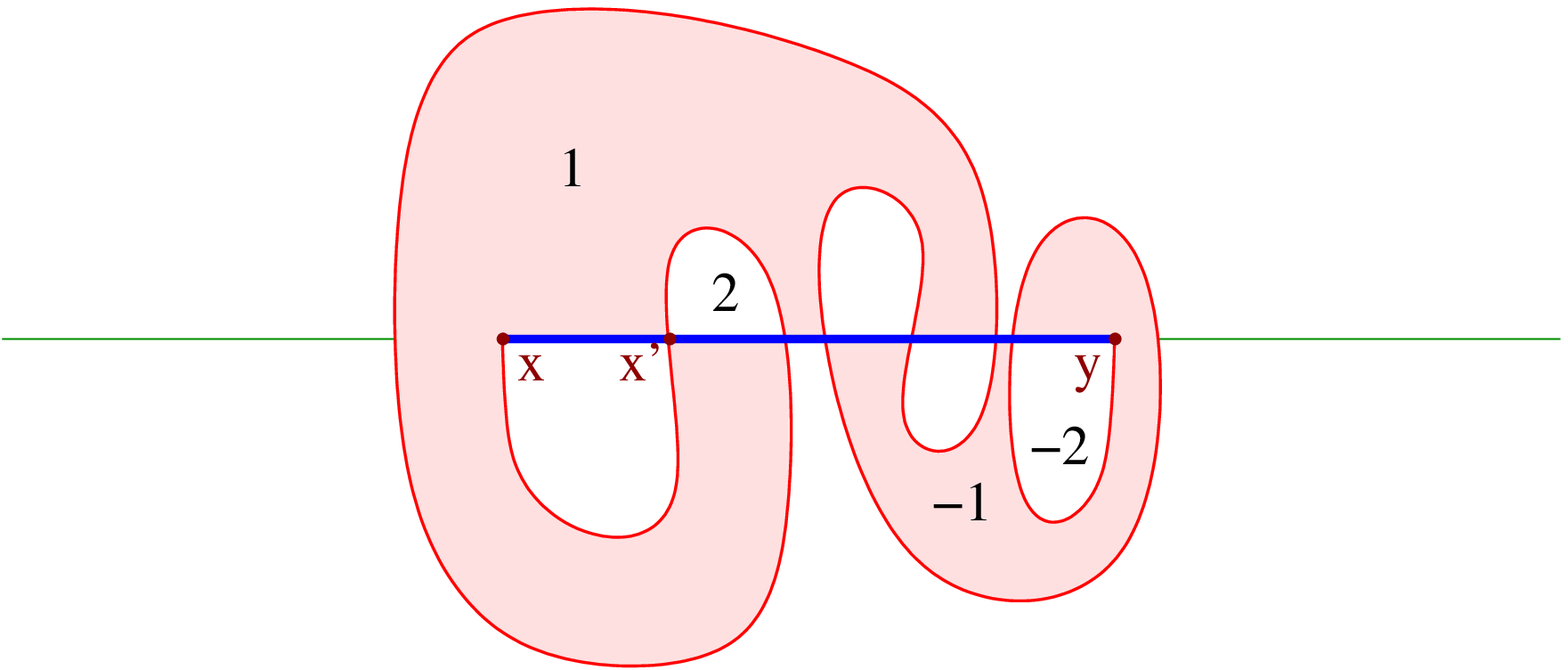} 
\caption{{Maslov index and catenation: $x<x'<y$.}}
\label{fig:maslov5}      
\end{figure}

\medskip\noindent
By Step~2 we may assume $\eps_x(\Lambda)=1$.
Since the Maslov index of the arc in $B$ 
from $x$ to $x'$ is $-1$, we have 
\begin{equation}\label{eq:case21}
\mu(\Lambda) = \mu(\Lambda')+1.
\end{equation}
Since the orientations of $A'$ and $B'$ 
agree with those of $A$ and $B$ we have
\begin{equation}\label{eq:case22}
\eps_{x'}(\Lambda') = \eps_{x'}(\Lambda)=-1,\qquad
\eps_y(\Lambda')=\eps_y(\Lambda).
\end{equation}
Now let $x<x_1<x_2<\cdots<x_m<x'$ be the intersection
points of $\R$ and $B$ in the interval $(x,x')$
and let $\eps_i\in\{-1,+1\}$ be the intersection
index of $\R$ and $B$ at $x_i$.  
Since the value of $\w$ slightly to the left of $x'$ 
agrees with the value of $\w$ slightly to the right of $x$ 
we have 
$$
\sum_{i=1}^m\eps_i=0. 
$$
Since $k_{x'}(\Lambda')$ is the sum of the intersection
indices of $\R$ and $B'$ at all points to the left of $x'$
we obtain
\begin{equation}\label{eq:case23}
k_{x'}(\Lambda') 
= k_x(\Lambda) + \sum_{i=1}^m\eps_i
= k_x(\Lambda),\qquad
k_y(\Lambda')=k_y(\Lambda).
\end{equation}
It follows from equation~\eqref{eq:masLov} for $\Lambda'$
and equations~\eqref{eq:case21}, \eqref{eq:case22}, 
and~\eqref{eq:case23} that
\begin{eqnarray*}
\mu(\Lambda) 
&=& 
\mu(\Lambda')+1 \\
&=& 
2k_{x'}(\Lambda')+2k_y(\Lambda')
+ \frac{\eps_{x'}(\Lambda')-\eps_y(\Lambda')}{2}
+ 1 \\
&=&
2k_x(\Lambda)+2k_y(\Lambda)
+ \frac{-1-\eps_y(\Lambda)}{2}
+ 1 \\
&=&
2k_x(\Lambda)+2k_y(\Lambda)
+ \frac{\eps_x(\Lambda)-\eps_y(\Lambda)}{2}.
\end{eqnarray*}
This proves Step~5.

\medskip\noindent{\bf Step~6.}
{\it Assume~\eqref{eq:AB} and $\#(B\cap\R)>2$, 
follow the arc of $B$, starting at $x$, 
and let~$y'$ be the next intersection point with $\R$. 
Assume $y'>y$. Denote by $B'$ the arc in $B$ from $y$ to $y'$, 
and let $A':=[y,y']$ (see Figure~\ref{fig:maslov6}).
If the $(\alpha,\beta)$-trace 
$\Lambda'$ with boundary $\p\Lambda'=(y,y',A',B')$
satisfies~\eqref{eq:masLov} so does $\Lambda$.}
\begin{figure}[htp] 
\centering 
\includegraphics[scale=0.4]{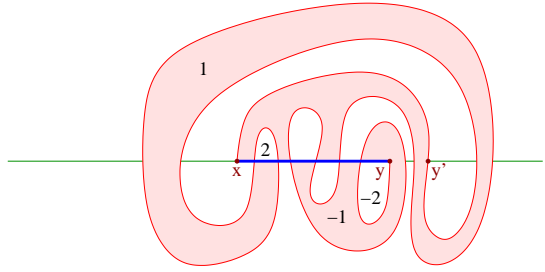} 
\caption{{Maslov index and catenation: $x<y<y'$.}}
\label{fig:maslov6}      
\end{figure}

\medskip\noindent
By Step~2 we may assume $\eps_x(\Lambda)=1$.
Since the orientation of $B'$ from $y$ to $y'$ is opposite
to the orientation of $B$ and the Maslov index 
of the arc in $B$ from $x$ to $y'$ is $-1$, we have 
\begin{equation}\label{eq:case31}
\mu(\Lambda) = 1-\mu(\Lambda').
\end{equation}
Using again the fact that the orientation of $B'$ is opposite
to the orientation of $B$ we have
\begin{equation}\label{eq:case32}
\eps_y(\Lambda') = -\eps_y(\Lambda),\qquad
\eps_{y'}(\Lambda')=-\eps_{y'}(\Lambda) = 1.
\end{equation}
Now let $x_1<x_2<\cdots<x_m$ be all 
intersection points of $\R$ and $B$ 
and let $\eps_i\in\{-1,+1\}$ be 
the intersection index of $\R$ and $B$ at $x_i$.  
Choose 
$$
j<k<\ell
$$ 
such that
$$
x_j=x,\qquad x_k=y,\qquad x_\ell=y'.
$$
Then 
$$
\eps_j=\eps_x(\Lambda)=1,\qquad
\eps_k=\eps_y(\Lambda),\qquad
\eps_\ell=\eps_{y'}(\Lambda)=-1, 
$$
and 
$$
k_x(\Lambda) = \sum_{i<j}\eps_i,\qquad
k_y(\Lambda) = - \sum_{i>k}\eps_i.
$$
For $i\ne j$ the intersection index
of $\R$ and $B'$ at $x_i$ is $-\eps_i$.
Moreover, $k_y(\Lambda')$ is the sum of the 
intersection indices of $\R$ and $B'$ 
at all points to the left of $y$ 
and $k_{y'}(\Lambda')$ is minus the sum of the 
intersection indices of $\R$ and $B'$ 
at all points to the right of $y'$. Hence
$$
k_y(\Lambda') 
= - \sum_{i<j}\eps_i
  - \sum_{j<i<k}\eps_i,\qquad
k_{y'}(\Lambda') 
= \sum_{i>\ell}\eps_i.
$$
We claim that 
\begin{equation}\label{eq:case33}
k_{y'}(\Lambda') + k_x(\Lambda) = 0,\qquad
k_y(\Lambda') + k_y(\Lambda) = \frac{1+\eps_y(\Lambda)}{2}.
\end{equation}
To see this, note that the value of the winding number 
$\w$ slightly to the left of $x$ agrees with the value
of $\w$ slightly to the right of $y'$, and hence
$$
0 = \sum_{i<j}\eps_i + \sum_{i>\ell}\eps_i 
= k_x(\Lambda) + k_{y'}(\Lambda').
$$
This proves the first equation in~\eqref{eq:case33}. 
To prove the second equation in~\eqref{eq:case33} 
we observe that
$$
\sum_{i=1}^m\eps_i = \frac{\eps_x(\Lambda)+\eps_y(\Lambda)}{2}
$$
and hence
\begin{eqnarray*}
k_y(\Lambda') + k_y(\Lambda)
&=&
- \sum_{i<j}\eps_i 
- \sum_{j<i<k}\eps_i
- \sum_{i>k}\eps_i  \\
&=&
\eps_j + \eps_k - \sum_{i=1}^m\eps_i \\
&=&
\eps_x(\Lambda) + \eps_y(\Lambda) - \sum_{i=1}^m\eps_i \\
&=& 
\frac{\eps_x(\Lambda)+\eps_y(\Lambda)}{2} \\
&=& 
\frac{1+\eps_y(\Lambda)}{2}.
\end{eqnarray*}
This proves the second equation in~\eqref{eq:case33}.

It follows from equation~\eqref{eq:masLov} for $\Lambda'$
and equations~\eqref{eq:case31}, \eqref{eq:case32}, 
and~\eqref{eq:case33} that
\begin{eqnarray*}
\mu(\Lambda) 
&=& 
1-\mu(\Lambda') \\
&=& 
1 - 2k_y(\Lambda')-2k_{y'}(\Lambda')
- \frac{\eps_y(\Lambda')-\eps_{y'}(\Lambda')}{2} \\
&=& 
1 - 2k_y(\Lambda')-2k_{y'}(\Lambda')
- \frac{-\eps_y(\Lambda)-1}{2} \\
&=& 
2k_y(\Lambda) - \eps_y(\Lambda) + 2k_x(\Lambda)
+ \frac{1+\eps_y(\Lambda)}{2} \\
&=&
2k_x(\Lambda)+2k_y(\Lambda) 
+ \frac{1-\eps_y(\Lambda)}{2}.
\end{eqnarray*}
Here the first equality follows from~\eqref{eq:case31}, 
the second equality follows from~\eqref{eq:masLov} for $\Lambda'$,
the third equality follows from~\eqref{eq:case32}, 
and the fourth equality follows from~\eqref{eq:case33}.
This proves Step~6.

\medskip\noindent{\bf Step~7.}
{\it Equation~\eqref{eq:maslov} holds when $\Sigma=\C$
and $\Lambda$ satisfies the arc condition.}

\medskip\noindent 
It follows from Steps~3-6 by induction that equation~\eqref{eq:masLov}
holds for every $(\alpha,\beta)$-trace $\Lambda=(x,y,\w)$
whose boundary $\p\Lambda=(x,y,A,B)$
satisfies~\eqref{eq:AB}.  Hence Step~7 follows from Step~1.

\medskip
Next we drop the assumption that $\Lambda$ satisfies the 
arc condition and extend the result to planar surfaces.  
This requires a further three steps.

\medskip\noindent{\bf Step~8.}
{\it Equation~\eqref{eq:maslov} holds when $\Sigma=\C$ and $x=y$.}

\medskip\noindent
Under these assumptions $\nu_\alpha:=\p\w|_{\alpha\setminus\beta}$ 
and $\nu_\beta:=-\p\w|_{\beta\setminus\alpha}$ are constant.
There are four cases.

\smallskip\noindent{\bf Case~1.} 
{\it $\alpha$ is an embedded circle
and $\beta$ is not an embedded circle.} 
In this case we have $\nu_\beta\equiv0$ and $B=\{x\}$.
Moroeover, $\alpha$ is the boundary of a unique disc $\Delta_\alpha$
and we assume that $\alpha$ is oriented as the boundary
of $\Delta_\alpha$. Then the path $\gamma_\alpha:[0,1]\to\Sigma$
in Definition~\ref{def:trace} satisfies 
$\gamma_\alpha(0)=\gamma_\alpha(1)=x$ and 
is homotopic to $\nu_\alpha\alpha$. Hence
$$
m_x(\Lambda) = m_y(\Lambda) = 2\nu_\alpha
= \mu(\Lambda).
$$
Here the last equation follows from the fact 
that $\Lambda$ can be obtained as the catenation 
of $\nu_\alpha$ copies of the disc $\Delta_\alpha$.

\smallskip\noindent{\bf Case~2.} 
{\it $\alpha$ is not an embedded circle
and $\beta$ is an embedded circle.} 
This follows from Case~1 by interchanging $\alpha$ and $\beta$.

\smallskip\noindent{\bf Case~3.} 
{\it $\alpha$ and $\beta$ are embedded circles.} 
In this case there is a unique pair of embedded 
discs $\Delta_\alpha$ and $\Delta_\beta$ with boundaries 
$\alpha$ and $\beta$, respectively. Orient $\alpha$
and $\beta$ as the boundaries of these discs. 
Then, for every $z\in\Sigma\setminus\alpha\cup\beta$, 
we have
$$
\w(z)=\left\{\begin{array}{ll}
\nu_\alpha-\nu_\beta,&\mbox{for } 
z\in\Delta_\alpha\cap\Delta_\beta,\\
\nu_\alpha,&\mbox{for } 
z\in\Delta_\alpha\setminus\overline{\Delta}_\beta,\\
-\nu_\beta,&
\mbox{for } z\in\Delta_\beta\setminus\overline{\Delta}_\alpha,\\
0,& \mbox{for } 
z\in\Sigma\setminus\overline{\Delta}_\alpha\cup\overline{\Delta}_\beta.
\end{array}\right.
$$
Hence
$$
m_x(\Lambda) = m_y(\Lambda) 
= 2\nu_\alpha-2\nu_\beta = \mu(\Lambda).
$$
Here the last equation follows from the fact 
$\Lambda$ can be obtained as the catenation 
of $\nu_\alpha$ copies of the disc $\Delta_\alpha$ 
(with the orientation inherited from $\Sigma$)
and $\nu_\beta$ copies of $-\Delta_\beta$
(with the opposite orientation).

\smallskip\noindent{\bf Case~4.} 
{\it Neither $\alpha$ nor $\beta$ is an embedded circle.} 
Under this assumption we have $\nu_\alpha=\nu_\beta=0$.
Hence it follows from Theorem~\ref{thm:trace}
that $\w=0$ and $\Lambda=\Lambda_u$ for the constant 
map $u\equiv x\in\cD(x,x)$.  Thus
$$
m_x(\Lambda)=m_y(\Lambda)=\mu(\Lambda)=0.
$$
This proves Step~8.

\medskip\noindent{\bf Step~9.}
{\it Equation~\eqref{eq:maslov} holds when $\Sigma=\C$.}

\medskip\noindent
By Step~8, it suffices to assume $x\ne y$.
It follows from Theorem~\ref{thm:trace} that every 
$u\in\cD(x,y)$ is homotopic to a catentation $u = u_0\#v$,
where $u_0\in\cD(x,y)$ satisfies the arc condition 
and $v\in\cD(y,y)$.  Hence it follows from Steps~7 and~8 
that
\begin{eqnarray*}
\mu(\Lambda_u)
&=&
\mu(\Lambda_{u_0}) + \mu(\Lambda_v) \\
&=&
\frac{m_x(\Lambda_{u_0})+m_y(\Lambda_{u_0})}{2} + m_y(\Lambda_v) \\
&=&
\frac{m_x(\Lambda_u)+m_y(\Lambda_u)}{2}.
\end{eqnarray*}
Here the last equation follows from the fact that 
$
\w_u=\w_{u_0}+\w_v
$ 
and hence 
$
m_z(\Lambda_u)=m_z(\Lambda_{u_0})+m_z(\Lambda_v)
$
for every $z\in\alpha\cap\beta$. This proves Step~9.  

\medskip\noindent{\bf Step~10.}
{\it Equation~\eqref{eq:maslov} holds when $\Sigma$ is planar.}

\medskip\noindent
Choose an element $u\in\cD(x,y)$ such that $\Lambda_u=\Lambda$.
Modifying $\alpha$ and $\beta$ on the complement of $u(\D)$,
if necessary, we may assume without loss of generality that
$\alpha$ and $\beta$ are mebedded circles.
Let $\iota:\Sigma\to\C$ be an orientation preserving embedding.
Then $\iota_*\Lambda := \Lambda_{\iota\circ u}$ is an 
$(\iota(\alpha),\iota(\beta))$-trace in $\C$ and hence 
satisfies~\eqref{eq:maslov} by Step~9.  Since 
$m_{\iota(x)}(\iota_*\Lambda)=m_x(\Lambda)$, 
$m_{\iota(y)}(\iota_*\Lambda)=m_y(\Lambda)$,
and $\mu(\iota_*\Lambda)=\mu(\Lambda)$ it follows 
that $\Lambda$ also satisfies~\eqref{eq:maslov}. 
This proves Step~10 
and Proposition~\ref{prop:maslovC}
\end{proof}

\begin{remark}\label{rmk:maslov}\rm
Let $\Lambda=(x,y,A,B)$ be an $(\alpha,\beta)$-trace in $\C$ 
as in Step~1 in the proof of Theorem~\ref{thm:maslov}.
Thus $x<y$ are real numbers, $A$ is the interval $[x,y]$, 
and $B$ is an embedded arc with endpoints $x,y$
which is oriented from $x$ to $y$ and is 
transverse to $\R$. Thus
$
Z:=B\cap\R
$
is a finite set.  Define a map
$$
f:Z\setminus\{y\}\to Z\setminus\{x\}
$$
as follows. Given $z\in Z\setminus\{y\}$ walk 
along $B$ towards $y$ and let $f(z)$ be the next 
intersection point with $\R$. This map is bijective.  
Now let $I$ be any of the three open intervals 
$(-\infty,x)$, $(x,y)$, $(y,\infty)$.  
Any arc in $B$ from $z$ to $f(z)$ with
both endpoints in the same interval $I$ 
can be removed by an isotopy of $B$ which
does not pass through $x,y$. 
Call $\Lambda$ a {\bf reduced $(\alpha,\beta)$-trace} 
if $z\in I$ implies $f(z)\notin I$ for each of the three intervals. 
Then every $(\alpha,\beta)$-trace is isotopic to a reduced 
$(\alpha,\beta')$-trace and the isotopy does not effect the 
numbers $\mu,k_x,k_y,\eps_x,\eps_y$.
\begin{figure}[htp]
\centering 
\includegraphics[scale=0.6]{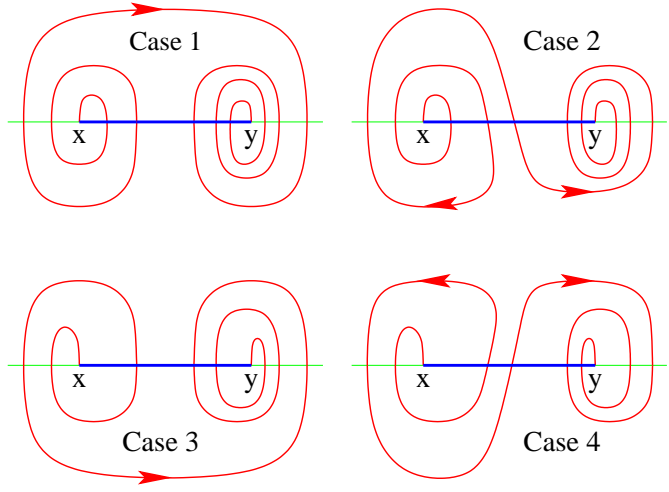} 
\caption{{Reduced $(\alpha,\beta)$-traces in $\C$.}}
\label{fig:maslov}
\end{figure}

Let $Z^+$ (respectively $Z^-$) denote the set 
of all points $z\in Z=B\cap\R$ where the positive 
tangent vectors in $T_zB$ point up (respectively down).
One can prove that every reduced $(\alpha,\beta)$-trace
satisfies one of the following conditions. 

\smallskip
\centerline{{\bf Case~1:} 
If $z\in Z^+\setminus\{y\}$ then $f(z)>z$.\qquad
{\bf Case~2:} $Z^-\subset[x,y]$.}

\centerline{{\bf Case~3:} 
If $z\in Z^-\setminus\{y\}$ then $f(z)>z$.\qquad
{\bf Case~4:} $Z^+\subset[x,y]$.}

\smallskip\noindent
(Examples with $\eps_x=1$ and $\eps_y=-1$ 
are depicted in Figure~\ref{fig:maslov}.)
One can then show directly that the reduced $(\alpha,\beta)$-traces
satisfy equation~\eqref{eq:masLov}.  This gives rise 
to an alternative proof of Proposition~\ref{prop:maslovC} 
via case distinction. 
\end{remark}

\begin{proof}[Proof of Theorem~\ref{thm:maslov}
in the Simply Connected Case]
If $\Sigma$ is diffeomorphic to the $2$-plane
the result has been established 
in Proposition~\ref{prop:maslovC}.
Hence assume 
$$
\Sigma=S^2. 
$$
Let $u\in\cD(x,y)$.  
If $u$ is not surjective the assertion follows from the case 
of the complex plane (Proposition~\ref{prop:maslovC})
via stereographic projection.  Hence assume $u$ is 
surjective and choose a regular value 
$z\in S^2\setminus(\alpha\cup\beta)$ of $u$.  
Denote 
$$
u^{-1}(z) = \{z_1,\dots,z_k\}.
$$
For $i=1,\dots,k$ let $\eps_i=\pm1$ according to whether or not
the differential $du(z_i):\C\to T_z\Sigma$ is orientation preserving.
Choose an open disc $\Delta\subset S^2$ centered at $z$
such that 
$$
\bar\Delta\cap(\alpha\cup\beta)=\emptyset
$$
and $u^{-1}(\Delta)$ is a union of open neighborhoods
$U_i\subset\D$ of $z_i$ with disjoint closures such that 
$$
u|_{U_i}:U_i\to\Delta
$$ 
is a diffeomorphism for each $i$
which extends to a neighborhood of $\bar U_i$. 
Now choose a continuous map $u':\D\to S^2$ which agrees
with $u$ on $\D\setminus\bigcup_iU_i$ and restricts to a 
diffeomorphism from $\bar U_i$ to $S^2\setminus\Delta$ 
for each $i$. Then $z$ does not belong to the image of $u'$
and hence equation~\eqref{eq:maslov} holds for $u'$
(after smoothing along the boundaries $\p U_i$).  
Moreover, the diffeomorphism 
$$
u'|_{\bar U_i}:\bar U_i\to S^2\setminus\Delta
$$
is orientation preserving if and only if $\eps_i=-1$.
Hence
\begin{equation*}
\begin{split}
\mu(\Lambda_u) &= \mu(\Lambda_{u'}) + 4\sum_{i=1}^k\eps_i,\\
m_x(\Lambda_u) &= m_x(\Lambda_{u'}) + 4\sum_{i=1}^k\eps_i,\\
m_y(\Lambda_u) &= m_y(\Lambda_{u'}) + 4\sum_{i=1}^k\eps_i.
\end{split}
\end{equation*}
By Proposition~\ref{prop:maslovC} equation~\eqref{eq:maslov} 
holds for $\Lambda_{u'}$ and hence it also holds for $\Lambda_u$.  
This proves Theorem~\ref{thm:maslov} when $\Sigma$ is simply
connected.
\end{proof}

%%%%%%%%%%%%%%%%%%%%%%%%%%%%%%%%%%%%%%%%%%%%%%%%
%%%%%%%%%%%%%%%%%%%%%%%%%%%%%%%%%%%%%%%%%%%%%%%%
%%%%%%%%%%%%%%%%%% Section 5 %%%%%%%%%%%%%%%%%%%
%%%%%%%%%%%%%%%%%%%%%%%%%%%%%%%%%%%%%%%%%%%%%%%%
%%%%%%%%%%%%%%%%%%%%%%%%%%%%%%%%%%%%%%%%%%%%%%%%

\section{The Non Simply Connected Case}\label{sec:LIFT}

The key step for extending Proposition~\ref{prop:maslovC} 
to non-simply connected two-manifolds is the next
result about lifts to the universal cover. 

\begin{proposition}\label{prop:gm}
Suppose $\Sigma$ is not diffeomorphic to the $2$-sphere.
Let ${\Lambda=(x,y,\w)}$ be an 
$(\alpha,\beta)$-trace and ${\pi:\C\to\Sigma}$ be a 
universal covering. Denote by $\Gamma\subset\Diff(\C)$
the group of deck transformations. 
Choose an element ${\tx\in\pi^{-1}(x)}$ and let $\tal$ 
and $\tbe$ be the lifts of $\alpha$ and $\beta$ 
through~$\tx$. Let $\tLa=(\tx,\ty,\tilde\w)$ be the lift
of $\Lambda$ with left endpoint~$\tx$. Then
\begin{equation}\label{eq:gm}
m_{g\tx}(\tLa) + m_{g^{-1}\ty}(\tLa) 
= 0
\end{equation}
for every $g\in\Gamma\setminus\{\id\}$.
\end{proposition}

\bigbreak

\begin{lemma}[{\bf Annulus Reduction}]\label{le:gm}
Suppose $\Sigma$ is not diffeomorphic to the $2$-sphere. 
Let $\Lambda$, $\pi$, $\Gamma$, $\tLa$
be as in Proposition~\ref{prop:gm}.  If 
\begin{equation}\label{eq:GM}
m_{g\tx}(\tLa) - m_{g\ty}(\tLa)
= m_{g^{-1}\ty}(\tLa) - m_{g^{-1}\tx}(\tLa)
\end{equation}
for every $g\in\Gamma\setminus\{\id\}$ 
then equation~\eqref{eq:gm} holds 
for every $g\in\Gamma\setminus\{\id\}$.
\end{lemma}

\begin{proof}
If~\eqref{eq:gm} does not hold then there is 
a deck transformation $h\in\Gamma\setminus\{\id\}$ 
such that
$
m_{h\tx}(\tLa) + m_{h^{-1}\ty}(\tLa) \ne 0.
$
Since there can only be finitely many such 
$h\in\Gamma\setminus\{\id\}$, 
there is an integer $k\ge 1$ such that 
$m_{h^k\tx}(\tLa)+m_{h^{-k}\ty}(\tLa)\ne0$
and $m_{h^\ell\tx}(\tLa)+m_{h^{-\ell}\ty}(\tLa)=0$ 
for every integer $\ell>k$. Define $g:=h^k$.  Then
\begin{equation}\label{eq:GM2}
m_{g\tx}(\tLa) + m_{g^{-1}\ty}(\tLa) \ne 0
\end{equation}
and $m_{g^k\tx}(\tLa) + m_{g^{-k}\ty}(\tLa)=0$
for every integer $k\in\Z\setminus\{-1,0,1\}$.
Define
$$
\Sigma_0:=\C/\Gamma_0,\qquad 
\Gamma_0 := \left\{g^k\,|\,k\in\Z\right\}.
$$
Then $\Sigma_0$ is diffeomorphic to the annulus.
Let $\pi_0:\C\to\Sigma_0$ be the obvious 
projection, define $\alpha_0:=\pi_0(\tal)$, $\beta_0:=\pi_0(\tbe)$,
and let $\Lambda_0:=(x_0,y_0,\w_0)$
be the $(\alpha_0,\beta_0)$-trace in $\Sigma_0$ with 
$x_0:=\pi_0(\tx)$, $y_0:=\pi_0(\ty)$, and
$$
\w_0(z_0) := \sum_{\tz\in\pi_0^{-1}(z_0)}\tilde\w(\tz),\qquad 
z_0\in\Sigma_0\setminus(\alpha_0\cup\beta_0).
$$
Then 
\begin{equation*}
\begin{split}
m_{x_0}(\Lambda_0) 
&= m_\tx(\tLa) + \sum_{k\in\Z\setminus\{0\}}m_{g^k\tx}(\tLa),\\
m_{y_0}(\Lambda_0) 
&= m_\ty(\tLa) + \sum_{k\in\Z\setminus\{0\}}m_{g^{-k}\ty}(\tLa).
\end{split}
\end{equation*}
By Proposition~\ref{prop:maslovC} both $\tLa$ 
and $\Lambda_0$ satisfy equation~\eqref{eq:maslov} 
and they have the same Viterbo--Maslov index. Hence
\begin{eqnarray*}
0 
&=&
\mu(\Lambda_0)-\mu(\tLa) \\
&=&
\frac{m_{x_0}(\Lambda_0) + m_{y_0}(\Lambda_0)}{2}
- \frac{m_\tx(\tLa)+m_\ty(\tLa)}{2} \\
&=&
\frac12
\sum_{k\ne 0}\left(m_{g^k\tx}(\tLa)+m_{g^{-k}\ty}(\tLa)\right) \\
&=&
m_{g\tx}(\tLa)+m_{g^{-1}\ty}(\tLa).
\end{eqnarray*}
Here the last equation follows from~\eqref{eq:GM}.
This contradicts~\eqref{eq:GM2} and proves Lemma~\ref{le:gm}.
\end{proof}

\bigbreak

\begin{lemma}\label{le:AB}
Suppose $\Sigma$ is not diffeomorphic to the $2$-sphere.
Let $\Lambda$, $\pi$, $\Gamma$, $\tLa$ be as in 
Proposition~\ref{prop:gm} and denote 
$\nu_\tal:=\p\tilde\w|_{\tal\setminus\tbe}$
and $\nu_\tbe:=-\p\tilde\w|_{\tbe\setminus\tal}$.
Choose smooth paths
$$
\gamma_\tal:[0,1]\to\tal,\qquad \gamma_\tbe:[0,1]\to\tbe
$$
from $\gamma_\tal(0)=\gamma_\tbe(0)=\tx$ to
$\gamma_\tal(1)=\gamma_\tbe(1)=\ty$ such that
$\gamma_\tal$ is an immersion when $\nu_\tal\not\equiv0$
and constant when $\nu_\tal\equiv0$, the same holds
for $\gamma_\tbe$, and
\begin{equation*}
\begin{split}
\nu_\tal(\tz) =\deg(\gamma_\tal,\tz)&\quad
\mbox{for}\quad \tz\in\tal\setminus\{\tx,\ty\},\\
\nu_\tbe(\tz) =\deg(\gamma_\tbe,\tz)&\quad
\mbox{for}\quad\tz\in\tbe\setminus\{\tx,\ty\}.
\end{split}
\end{equation*}
Define 
$$
\tA := \gamma_\tal([0,1]),\qquad \tB := \gamma_\tbe([0,1]).
$$
Then, for every $g\in\Gamma$, we have
\begin{equation}\label{eq:gxyA1}
g\tx\in\tA
\qquad\iff\qquad 
g^{-1}\ty\in\tA,
\end{equation}
\begin{equation}\label{eq:gxyA2}
g\tx\notin\tA\;\;\mbox{ and }\;\;g\ty\notin\tA 
\qquad\iff\qquad 
\tA\cap g\tA=\emptyset,
\end{equation}
\begin{equation}\label{eq:gxyA3}
g\tx\in\tA\;\;\mbox{ and }\;\;g\ty\in\tA 
\qquad\iff\qquad g=\id.
\end{equation}
The same holds with $\tA$ replaced by $\tB$.
\end{lemma}

\begin{proof}
If $\alpha$ is a contractible embedded circle or 
not an embedded circle at all we have 
$\tA\cap g\tA=\emptyset$ whenever $g\ne\id$
and this implies~\eqref{eq:gxyA1}, \eqref{eq:gxyA2}
and~\eqref{eq:gxyA3}. Hence assume 
$\alpha$ is a noncontractible embedded circle.
Then we may also assume, 
without loss of generality, 
that $\pi(\R)=\alpha$, 
the map $\tz\mapsto\tz+1$ is a deck transformation, 
$\pi$ maps the interval $[0,1)$ bijectively onto~$\alpha$,
and $\tx,\ty\in\R=\tal$ with $\tx<\ty$.  
Thus $\tA=[\tx,\ty]$ and, for every $k\in\Z$, 
$$
\tx+k\in[\tx,\ty]
\quad\iff\quad
0\le k\le \ty-\tx
\quad\iff\quad
\ty-k\in[\tx,\ty].
$$
Similarly, we have 
$$
\tx+k,\ty+k\notin[\tx,\ty]
\quad\iff\quad
[\tx+k,\ty+k]\cap[\tx,\ty]=\emptyset
$$
and
$$
\tx+k,\ty+k\in[\tx,\ty]
\quad\iff\quad
[\tx+k,\ty+k]\subset[\tx,\ty]
\quad\iff\quad k=0.
$$
This proves~\eqref{eq:gxyA1}, \eqref{eq:gxyA2}, 
and~\eqref{eq:gxyA3} for the deck transformation 
$\tz\mapsto\tz+k$.  If $g$ is any other deck transformation, 
then we have 
$$
\tal\cap g\tal=\emptyset
$$ 
and so~\eqref{eq:gxyA1}, \eqref{eq:gxyA2}, and~\eqref{eq:gxyA3}
are trivially satisfied. This proves Lemma~\ref{le:AB}.
\end{proof}

\begin{lemma}[{\bf Winding Number Comparison}]\label{le:wg}
Suppose $\Sigma$ is not diffeomorphic to the $2$-sphere.
Let $\Lambda$, $\pi$, $\Gamma$, $\tLa$
be as in Proposition~\ref{prop:gm}, and let
$\tA,\tB\subset\C$ be as in Lemma~\ref{le:AB}.
Then the following holds.

\smallskip\noindent{\bf (i)}
Equation~\eqref{eq:GM} holds for every $g\in\Gamma$
that satisfies $g\tx,g\ty\notin\tA\cup\tB$.

\smallskip\noindent{\bf (ii)}
If $\Lambda$ satisfies the arc condition then~\eqref{eq:gm}
holds for every $g\in\Gamma\setminus\{\id\}$.
\end{lemma}

\begin{proof}
We prove~(i).  Let $g\in\Gamma$ such that $g\tx,g\ty\notin\tA\cup\tB$ 
and let $\gamma_\tal,\gamma_\tbe$ be as in Lemma~\ref{le:AB}.  
Then $\tilde\w(\tz)$ is the winding number of the loop 
$\gamma_\tal-\gamma_\tbe$ about the point 
$\tz\in\C\setminus(\tA\cup\tB)$. Moreover, the paths 
$$
g\gamma_\tal:[0,1]\to\C,\qquad
g\gamma_\tbe:[0,1]\to\C
$$ 
connect the points $g\tx,g\ty\in\C\setminus(\tA\cup\tB)$.
Hence
$$
\tilde\w(g\ty)-\tilde\w(g\tx)
= (\gamma_\tal-\gamma_\tbe)\cdot g\gamma_\tal
= (\gamma_\tal-\gamma_\tbe)\cdot g\gamma_\tbe.
$$
Similarly with $g$ replaced by $g^{-1}$. 
Moreover, it follows from Lemma~\ref{le:AB}, that
$$
\tA\cap g\tA =\emptyset,\qquad \tB\cap g^{-1}\tB=\emptyset.
$$
Hence
\begin{eqnarray*}
\tilde\w(g\ty) - \tilde\w(g\tx)
&=&  
\left(\gamma_\tal-\gamma_\tbe\right)\cdot g\gamma_\tal \\
&=&  
g\gamma_\tal\cdot\gamma_\tbe \\
&=&
\gamma_\tal\cdot g^{-1}\gamma_\tbe \\
&=&
\left(\gamma_\tal-\gamma_\tbe\right)\cdot g^{-1}\gamma_\tbe \\
&=&
\tilde\w(g^{-1}\ty) - \tilde\w(g^{-1}\tx)
\end{eqnarray*}
Here we have used the fact that every $g\in\Gamma$
is an orientation preserving diffeomorphism
of $\C$. Thus we have proved that
$$
\tilde\w(g\tx) + \tilde\w(g^{-1}\ty)
= \tilde\w(g\ty) + \tilde\w(g^{-1}\tx).
$$
Since $g\tx,g\ty\notin\tA\cup\tB$, 
we have 
$$
m_{g\tx}(\tLa)=4\tilde\w(g\tx),\qquad
m_{g^{-1}\ty}(\tLa)=4\tilde\w(g^{-1}\ty),
$$
and the same identities hold with $g$ replaced by $g^{-1}$.
This proves~(i).

We prove~(ii).
If $\Lambda$ satisfies the arc condition then $g\tA\cap\tA=\emptyset$
and $g\tB\cap\tB=\emptyset$ for every $g\in\Gamma\setminus\{\id\}$.
In particular, for every $g\in\Gamma\setminus\{\id\}$,
we have $g\tx,g\ty\notin\tA\cup\tB$ and hence~\eqref{eq:GM}
holds by~(i).  Hence it follows from Lemma~\ref{le:gm} 
that~\eqref{eq:gm} holds for every $g\in\Gamma\setminus\{\id\}$.
This proves Lemma~\ref{le:wg}. 
\end{proof}

The next lemma deals with $(\alpha,\beta)$-traces
connecting a point $x\in\alpha\cap\beta$ to itself. 
An example on the annulus is depicted in 
Figure~\ref{fig:annulus3}.

\begin{lemma}[{\bf Isotopy Argument}]\label{le:zero}
Suppose $\Sigma$ is not diffeomorphic to the $2$-sphere.
Let $\Lambda$, $\pi$, $\Gamma$, $\tLa$
be as in Proposition~\ref{prop:gm}.  Suppose that there 
is a deck transformation $g_0\in\Gamma\setminus\{\id\}$
such that $\ty = g_0\tx$.  Then $\Lambda$ has Viterbo--Maslov 
index zero and $m_{g\tx}(\tLa)=0$ for every 
$g\in\Gamma\setminus\{\id,g_0\}$.
\end{lemma}

\begin{figure}[htp]
\centering 
\includegraphics[scale=0.275]{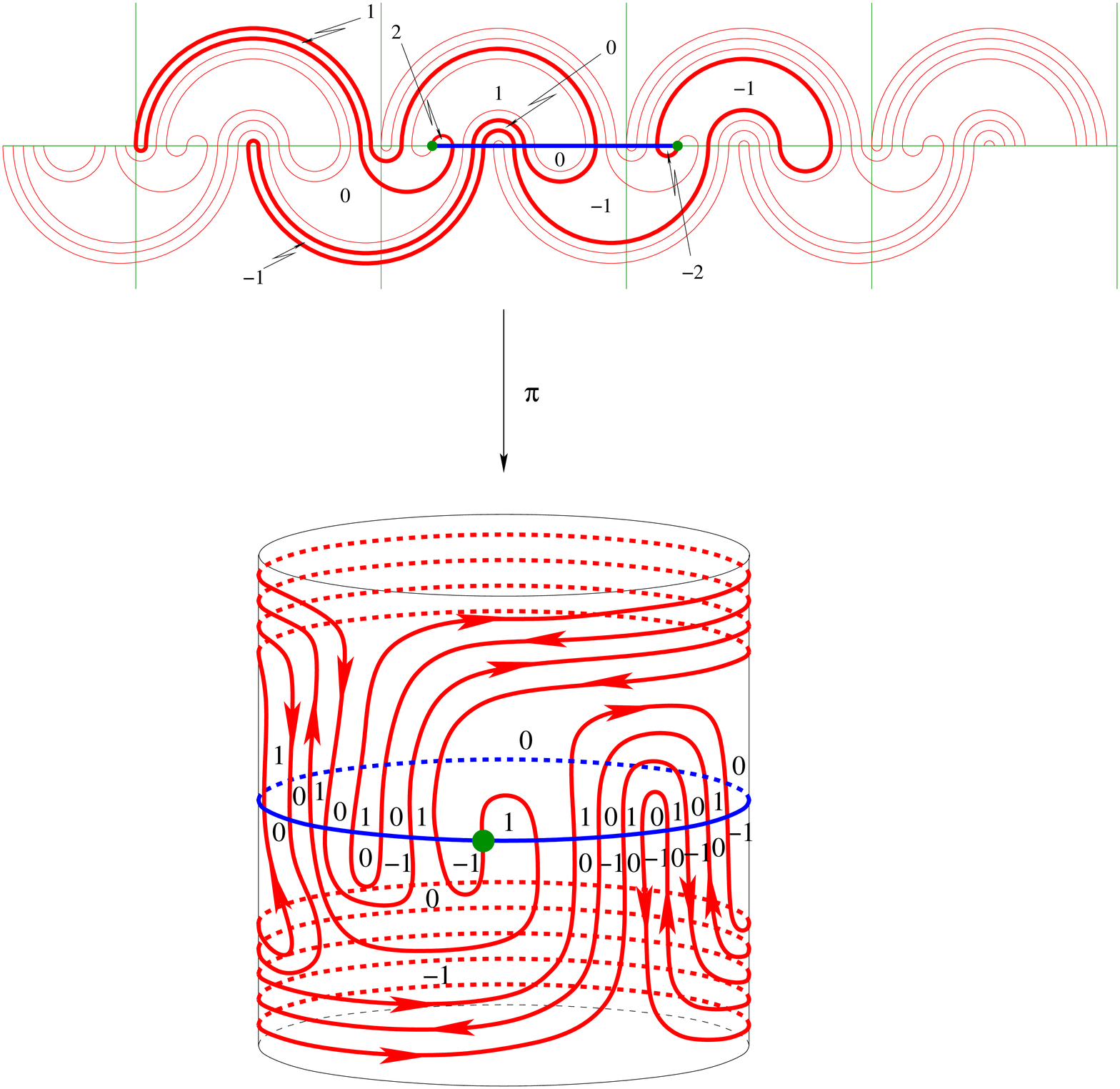} 
\caption{{An $(\alpha,\beta)$-trace on the annulus 
with $x=y$.}}
\label{fig:annulus3}
\end{figure}

\begin{proof}
By assumption, we have $\tal=g_0\tal$
and $\tbe=g_0\tbe$.  Hence $\alpha$ and $\beta$
are noncontractible embedded circles and some iterate 
of $\alpha$ is homotopic to some iterate of $\beta$.
Hence, by Lemma~\ref{le:dbae3}, $\alpha$ must 
be homotopic to $\beta$ (with some orientation).
Hence we may assume, without loss of generality, 
that $\pi(\R)=\alpha$, the map $\tz\mapsto\tz+1$ is a deck transformation, 
$\pi$ maps the interval $[0,1)$ bijectively onto~$\alpha$,
$\R=\tal$, $\tx=0\in\tal\cap\tbe$, $\tbe=\tbe+1$,
and that $\ty=\ell>0$ is an integer.  
Then $g_0$ is the translation 
$$
g_0(\tz)=\tz+\ell.
$$ 
Let $\tA:=[0,\ell]\subset\tal$ and let
$\tB\subset\tbe$ be the arc connecting $0$ to $\ell$.
Then, for $\tz\in\C\setminus(\tA\cup\tB)$, the integer
$\tilde\w(\tz)$ is the winding number of $\tA-\tB$ about $\tz$.
Define the projection $\pi_0:\C\to\C$ by
$$
\pi_0(\tz):=e^{2\pi\i\tz/k},
$$
denote $\alpha_0:=\pi_0(\tal)=S^1$ and $\beta_0:=\pi(\tbe)$,
and let 
$
\Lambda_0=(1,1,\w_0)
$
be the induced $(\alpha_0,\beta_0)$-trace in $\C$ with
$
\w_0(z) := \sum_{\tz\in\pi^{-1}(z)}\tilde\w(\tz).
$
Then $\alpha_0$ and $\beta_0$ are embedded circles
and have the winding number $\ell$ about zero.
Hence it follows from Step~8, Case~3 in the proof 
of Proposition~\ref{prop:maslovC} that $\Lambda_0$ 
has Viterbo--Maslov index zero and satisfies 
$m_{x_0}(\Lambda_0)+m_{y_0}(\Lambda_0)=2\mu(\Lambda_0)=0$.
Hence $\tLa$ also has Viterbo--Maslov index zero.

It remains to prove that $m_{g\tx}(\tLa)=0$ for every 
$g\in\Gamma\setminus\{\id,g_0\}$.  To see this we use 
the fact that the embedded loops $\alpha$
and $\beta$ are homotopic with fixed endpoint $x$.
Hence, by a Theorem of Epstein, they are
isotopic with fixed basepoint $x$
(see~\cite[Theorem~4.1]{EPSTEIN}).  
Thus there exists a smooth map 
$f:\R/\Z\times[0,1]\to\Sigma$ such that
$$
f(s,0)\in\alpha,\qquad f(s,1)\in\beta,\qquad
f(0,t)=x,
$$
for all $s\in\R/\Z$ and $t\in[0,1]$,
and the map $\R/\Z\to\Sigma:s\mapsto f(s,t)$ 
is an embedding for every $s\in[0,1]$.
Lift this homotopy to the universal cover 
to obtain a map $\tf:\R\times[0,1]\to\C$
such that $\pi\circ\tf = f$ and
$$
\tf(s,0)\in[0,1],\quad
\tf(s,1)\in\tB_1,\quad
\tf(0,t)=\tx,\quad
\tf(s+1,t)=\tf(s,t)+1
$$
for all $s\in\R$ and $t\in[0,1]$.  
Here $\tB_1\subset\tB$ denotes the arc in $\tB$ from $0$ to $1$.
Since the map $\R/\Z\to\Sigma:s\mapsto f(s,t)$ is injective
for every $t$, we have 
$$
g\tx\notin\{\tx,\tx+1,\dots,\tx+\ell\}\qquad\implies\qquad
g\tx \notin \tf([0,\ell]\times[0,1])
$$
for every every $g\in\Gamma$. Now choose a smooth map 
$\tu:\D\to\C$ with $\Lambda_\tu=\tLa$ (see Theorem~\ref{thm:trace}).  
Define the homotopy $F_\tu:[0,\ell]\times[0,1]\to\C$
by $F_\tu(s,t):=\tu(-\cos(\pi s/\ell),t\sin(\pi s/\ell))$.
Then, by Theorem~\ref{thm:trace}, 
$F_\tu$ is homotopic to $\tf|_{[0,\ell]\times[0,1]}$
subject to the boundary conditions $\tf(s,0)\in\tal=\R$,
$\tf(s,1)\in\tbe$, $\tf(0,t)=\tx$, $\tf(\ell,t)=\ty$.
Hence, for every $\tz\in\C\setminus(\tal\cup\tbe)$, we have 
$$
\tilde\w(\tz) = \deg(\tu,z) = \deg(F_\tu,\tz) = \deg(\tf,\tz).
$$
In particular, choosing $\tz$ near $g\tx$, we find
$
m_{g\tx}(\tLa) = 4\deg(\tf,g\tx) = 0
$
for every $g\in\Gamma$ that is not one of the translations 
$\tz\mapsto\tz+k$ for $k=0,1,\dots,\ell$.  This proves the
assertion in the case $\ell=1$.  

If $\ell>1$ it remains to prove $m_k(\tLa)=0$ for $k=1,\dots,\ell-1$.
To see this, let $\tA_1:=[0,1]$, $\tB_1\subset\tB$ 
be the arc from $0$ to $1$, $\tilde\w_1(\tz)$ be the winding 
number of $\tA_1-\tB_1$ about $\tz\in\C\setminus(\tA_1\cup\tB_1)$,
and define 
$
\tLa_1 := (0,1,\tilde\w_1).
$
Then, by what we have already proved, the 
$(\tal,\tbe)$-trace $\tLa_1$ satisfies $m_{g\tx}(\tLa_1)=0$
for every $g\in\Gamma$ other than the translations by $0$ or $1$.
In particular, we have $m_j(\tLa_1)=0$ for every $j\in\Z\setminus\{0,1\}$
and also $m_0(\tLa_1)+m_1(\tLa_1)=2\mu(\tLa_1)=0$.
Since $\tilde\w(\tz)=\sum_{j=0}^{\ell-1}\tilde\w_1(\tz-j)$
for $\tz\in\C\setminus(\tA\cup\tB)$, we obtain
$$
m_k(\tLa) = \sum_{j=0}^{\ell-1}m_{k-j}(\tLa_1) = 0
$$
for every $k\in\Z\setminus\{0,\ell\}$. 
This proves Lemma~\ref{le:zero}.
\end{proof}

The next example shows that Lemma~\ref{le:wg} cannot 
be strengthened to assert the identity $m_{g\tx}(\tLa)=0$ 
for every $g\in\Gamma$ with $g\tx,g\ty\notin\tA\cup\tB$.

\begin{example}\label{ex:annulus1}\rm
Figure~\ref{fig:annulus1} depicts an $(\alpha,\beta)$-trace 
$\Lambda=(x,y,\w)$ on the annulus $\Sigma=\C/\Z$ 
that has Viterbo--Maslov index one and satisfies the arc condition. 
The lift satisfies $m_\tx(\tLa)=-3$,  $m_{\tx+1}(\tLa)=4$, 
$m_{\ty}(\tLa)=5$, and $m_{\ty-1}(\tLa)=-4$. 
Thus $m_x(\Lambda)=m_y(\Lambda)=1$.
\begin{figure}[htp]
\centering 
\includegraphics[scale=0.4]{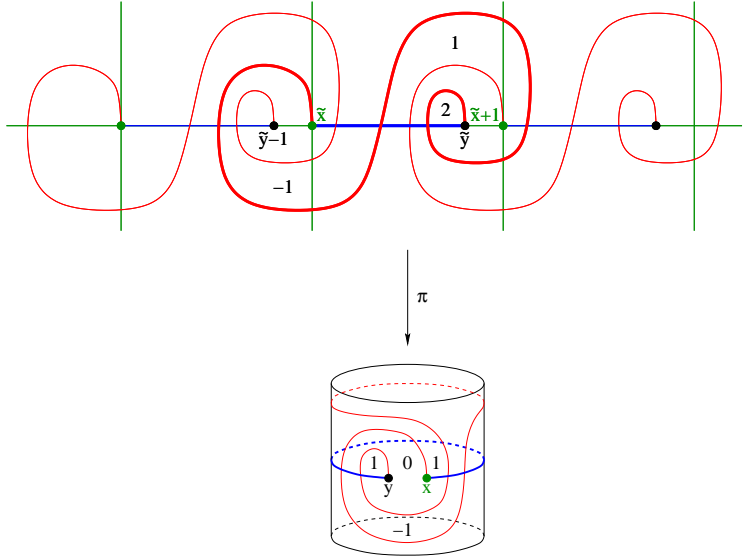} 
\caption{{An $(\alpha,\beta)$-trace on the annulus 
satisfying the arc condition.}}
\label{fig:annulus1}
\end{figure}
\end{example}

\begin{proof}[Proof of Proposition~\ref{prop:gm}]
The proof has five steps.

\medskip\noindent{\bf Step~1.}
{\it Let $\tA,\tB\subset\C$ be as in Lemma~\ref{le:AB}
and let $g\in\Gamma$ such that
$$
g\tx\in\tA\setminus\tB,\qquad g\ty\notin\tA\cup\tB.
$$ 
(An example is depicted in Figure~\ref{fig:torus}.)
Then~\eqref{eq:GM} holds.}

\begin{figure}[htp]
\centering 
\includegraphics[scale=0.2]{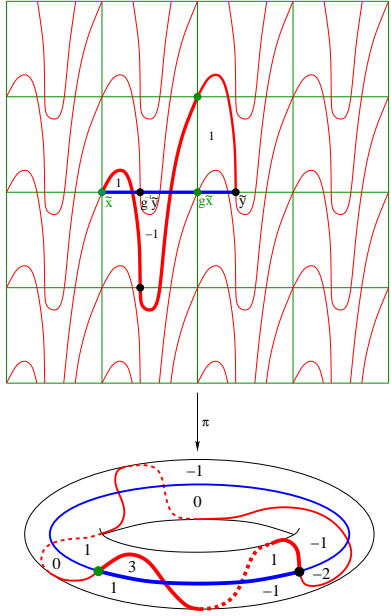} 
\caption{{An $(\alpha,\beta)$-trace on the torus 
not satisfying the arc condition.}}
\label{fig:torus}
\end{figure}

\medskip\noindent
The proof is a refinement of the winding number 
comparison argument in Lemma~\ref{le:wg}.
Since $g\tx\notin\tB$ we have $g\ne\id$ and, 
since $\tx,g\tx\in\tA\subset\tal$,
it follows that $\alpha$ is a noncontractible embedded circle.
Hence we may choose the universal covering $\pi:\C\to\Sigma$
and the lifts $\tal$, $\tbe$, $\tLa$ such that $\pi(\R)=\alpha$, 
the map $\tz\mapsto\tz+1$ is a deck transformation, 
the projection $\pi$ maps the interval $[0,1)$ 
bijectively onto $\alpha$, and
$$
\tal=\R,\qquad \tx=0\in\tal\cap\tbe,\qquad \ty>0.
$$
By assumption and Lemma~\ref{le:AB} 
there is an integer $k$ such that
$$
0<k<\ty,\qquad g\tx=k,\qquad g^{-1}\ty = \ty-k.
$$
Thus $g$ is the deck transformation $\tz\mapsto\tz+k$.

\smallbreak

Since $g\tx\notin\tB$ and $g\ty\notin\tB$ it follows from 
Lemma~\ref{le:AB} that $g^{-1}\ty\notin\tB$ 
and $g^{-1}\tx\notin\tB$ and hence, again
by Lemma~\ref{le:AB}, we have
$$
\tB\cap g\tB=\tB\cap g^{-1}\tB=\emptyset.
$$ 
With $\gamma_\tal$ and $\gamma_\tbe$ chosen as in 
Lemma~\ref{le:AB}, this implies
\begin{equation}\label{eq:zero}
\gamma_\tbe\cdot(\gamma_\tbe-k)
= (\gamma_\tbe+k)\cdot\gamma_\tbe
= 0.
\end{equation}
Since $k,-k,\ty+k,\ty-k\notin\tB$, there exists
a constant $\eps>0$ such that
$$
-\eps\le t\le\eps\qquad\implies\qquad
k+\i t,\;\;-k+\i t,\;\;
\ty-k+\i t,\;\;\ty+k+\i t\notin\tB.
$$
The paths $g\gamma_\tal\pm\i\eps$ 
and $g\gamma_\tbe\pm\i\eps$ both
connect the point $g\tx\pm\i\eps$ to $g\ty\pm\i\eps$.
Likewise, the paths $g^{-1}\gamma_\tal\pm\i\eps$ 
and $g^{-1}\gamma_\tbe\pm\i\eps$ both
connect the point $g^{-1}\tx\pm\i\eps$ to $g^{-1}\ty\pm\i\eps$.
Hence
\begin{eqnarray*}
\tilde\w(g\ty\pm\i\eps)
- \tilde\w(g\tx\pm\i\eps)
&=&
(\gamma_\tal-\gamma_\tbe)\cdot(g\gamma_\tal\pm\i\eps) \\
&=&
(\gamma_\tal-\gamma_\tbe)\cdot(\gamma_\tal+k\pm\i\eps) \\
&=&
(\gamma_\tal+k\pm\i\eps)\cdot\gamma_\tbe \\
&=&
\gamma_\tal\cdot(\gamma_\tbe-k\mp\i\eps) \\
&=&
(\gamma_\tal-\gamma_\tbe)\cdot(\gamma_\tbe-k\mp\i\eps) \\
&=&
(\gamma_\tal-\gamma_\tbe)\cdot(g^{-1}\gamma_\tbe\mp\i\eps) \\
&=&
\tilde\w(g^{-1}\ty\mp\i\eps)
- \tilde\w(g^{-1}\tx\mp\i\eps).
\end{eqnarray*}
Here the last but one equation follows from~\eqref{eq:zero}.
Thus we have proved
\begin{equation}\label{eq:step1}
\begin{split}
\tilde\w(g\tx+\i\eps) + \tilde\w(g^{-1}\ty-\i\eps)
&= \tilde\w(g^{-1}\tx-\i\eps) + \tilde\w(g\ty+\i\eps), \\
\tilde\w(g\tx-\i\eps) + \tilde\w(g^{-1}\ty+\i\eps)
&= \tilde\w(g^{-1}\tx+\i\eps) + \tilde\w(g\ty-\i\eps).
\end{split}
\end{equation}
Since
\begin{equation*}
\begin{split}
m_{g\tx}(\tLa) &= 2\tilde\w(g\tx+\i\eps) 
+ 2\tilde\w(g\tx-\i\eps), \\
m_{g\ty}(\tLa) &= 2\tilde\w(g\ty+\i\eps) 
+ 2\tilde\w(g\ty-\i\eps), \\
m_{g^{-1}\tx}(\tLa) &= 2\tilde\w(g^{-1}\tx+\i\eps) 
+ 2\tilde\w(g^{-1}\tx-\i\eps),\\
m_{g^{-1}\ty}(\tLa) &= 2\tilde\w(g^{-1}\ty+\i\eps) 
+ 2\tilde\w(g^{-1}\ty-\i\eps),
\end{split}
\end{equation*}
Step~1 follows by taking the sum 
of the two equations in~\eqref{eq:step1}.

\medskip\noindent{\bf Step~2.}
{\it Let $\tA,\tB\subset\C$ be as in Lemma~\ref{le:AB}
and let $g\in\Gamma$.  Suppose that either
$g\tx,g\ty\notin\tA$ or $g\tx,g\ty\notin\tB$.
Then~\eqref{eq:GM} holds.}

\medskip\noindent
If $g\tx,g\ty\notin\tA\cup\tB$ the assertion 
follows from Lemma~\ref{le:wg}.
If $g\tx\in\tA\setminus\tB$ and $g\ty\notin\tA\cup\tB$
the assertion follows from Step~1.
If $g\tx\notin\tA\cup\tB$ and $g\ty\in\tA\setminus\tB$
the assertion follows from Step~1 by interchanging $\tx$
and $\ty$.   Namely, \eqref{eq:GM} 
holds for $\tLa$ if and only if it holds for
the $(\tal,\tbe)$-trace
$
-\tLa:=(\ty,\tx,-\tilde\w).
$
This covers the case $g\tx,g\ty\notin\tB$. 
If $g\tx,g\ty\notin\tA$ the assertion follows 
by interchanging $\tA$ and $\tB$.
Namely, \eqref{eq:GM} 
holds for $\tLa$ if and only if it holds for
the $(\tbe,\tal)$-trace
$
\tLa^*:=(\tx,\ty,-\tilde\w).
$
This proves Step~2.

\medskip\noindent{\bf Step~3.}
{\it Let $\tA,\tB\subset\C$ be as in Lemma~\ref{le:AB}
and let $g\in\Gamma$ such that
$$
g\tx\in\tA\setminus\tB,\qquad g\ty\in\tB\setminus\tA.
$$ 
(An example is depicted in Figure~\ref{fig:annulus2}.)
Then~\eqref{eq:gm} holds for $g$ and $g^{-1}$.}

\begin{figure}[htp]
\centering 
\includegraphics[scale=0.28]{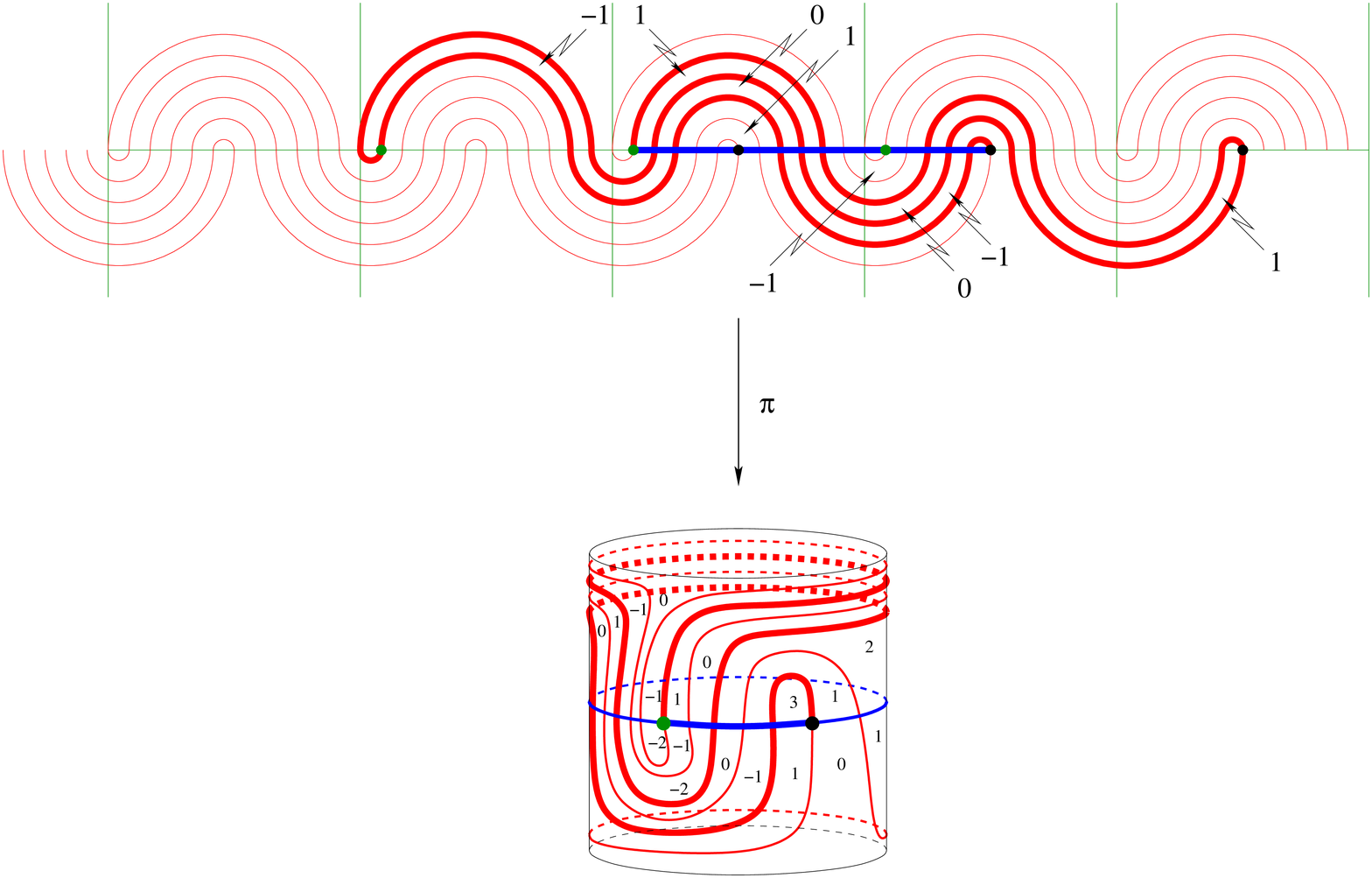} 
\caption{{An $(\alpha,\beta)$-trace on the annulus
with $g\tx\in\tA$ and $g\ty\in\tB$.}}
\label{fig:annulus2}
\end{figure}

\medskip\noindent
Since $g\tx\notin\tB$ 
(and $g\ty\notin\tA$) we have $g\ne\id$ and, 
since $\tx,g\tx\in\tA\subset\tal$ and $\ty,g\ty\in\tB\subset\tbe$,
it follows that $g\tal=\tal$ and $g\tbe=\tbe$.
Hence $\alpha$ and $\beta$ are noncontractible embedded circles
and some iterate of $\alpha$ is homotopic to some iterate of $\beta$.
So $\alpha$ is homotopic to~$\beta$ (with some orientation), 
by Lemma~\ref{le:dbae3}.
Hence we may choose the universal covering $\pi:\C\to\Sigma$
and the lifts ${\tal,\tbe,\tLa}$ such that $\pi(\R)=\alpha$, 
the map $\tz\mapsto\tz+1$ is a deck transformation, 
$\pi$~maps the interval $[0,1)$ bijectively onto $\alpha$, 
and $\tal=\R$, $\tx=0\in\tal\cap\tbe$, $\ty>0$.
Thus $\tA=[0,\ty]$ is the arc in $\tal$ 
from $0$ to $\ty$ and $\tB$ is the arc in $\tbe$ 
from $0$ to $\ty$. Moreover, $\tbe=\tbe+1$ and the arc 
in $\tbe$ from $0$ to $1$ is a fundamental domain 
for $\beta$.  By assumption and Lemma~\ref{le:AB} 
there is an integer $k$ such that $k\in\tA$ and $-k\in\tB$.
Hence $\tA$ does not contain any negative integers and
$\tB$ does not contain any positive integers.
Choose $k_\tA,k_\tB\in\N$ such that
$$
\tA\cap\Z=\left\{0,1,2,\cdots,k_\tA\right\},\qquad
\tB\cap\Z=\left\{0,-1,-2,\cdots,-k_\tB\right\}.
$$
For $0\le k\le k_\tA$ let $\tA_k\subset\tal$ and 
$\tB_k\subset\tbe$ be the arcs from $0$ to $\ty-k$
and consider the $(\tal,\tbe)$-trace
$$
\tLa_k:=(0,\ty-k,\tilde\w_k),\qquad
\p\tLa_k:=(0,\ty-k,\tA_k,\tB_k),
$$
where $\tilde\w_k(\tz)$ is the winding number of 
$\tA_k-\tB_k$ about $\tz\in\C\setminus(\tA_k\cup\tB_k)$. 
Note that $\tLa_0=\tLa$ and 
$$
\tB_k\cap\Z = \left\{0,-1,-2,\cdots,-k_\tB-k\right\}.
$$
We prove that, for each $k$,
the $(\tal,\tbe)$-trace $\tLa_k$ satisfies 
\begin{equation}\label{eq:jmk}
m_j(\tLa_k) + m_{\ty-k-j}(\tLa_k) = 0\qquad
\forall\,j\in\Z\setminus\{0\}.
\end{equation} 
If $\ty$ is an integer, then~\eqref{eq:jmk} 
follows from Lemma~\ref{le:zero}. 
Hence we may assume that $\ty$ is not an integer. 

We prove equation~\eqref{eq:jmk} by reverse induction on $k$.
First let $k=k_\tA$.  Then we have $j,\ty+j\notin\tA_k$ 
for every $j\in\N$.  Hence it follows from 
Step~2 that
\begin{equation}\label{eq:JMK}
m_j(\tLa_k) + m_{\ty-k-j}(\tLa_k) = 
m_{-j}(\tLa_k) + m_{\ty-k+j}(\tLa)
\qquad\forall\,j\in\N.
\end{equation} 
Thus we can apply Lemma~\ref{le:gm}
to the projection of $\tLa_k$ to the quotient $\C/\Z$.
Hence $\tLa_k$ satisfies~\eqref{eq:jmk}.

Now fix an integer $k\in\{0,1,\dots,k_\tA-1\}$ and suppose,
by induction, that $\tLa_{k+1}$ satisfies~\eqref{eq:jmk}. 
Denote by $\tA'\subset\tal$ and $\tB'\subset\tbe$
the arcs from $\ty-k-1$ to $1$, and by
$\tA''\subset\tal$ and $\tB''\subset\tbe$
the arcs from $1$ to $\ty-k$.  
Then $\tLa_k$ is the catenation 
of the $(\tal,\tbe)$-traces
\begin{equation*}
\begin{split}
\tLa_{k+1} := (0,\ty-k-1,\tilde\w_{k+1}),\quad
&\p\tLa_{k+1} = (0,\ty-k-1,\tA_{k+1},\tB_{k+1}),\\
\tLa' := (\ty-k-1,1,\tilde\w'),\quad
&\p\tLa' = (\ty-k-1,1,\tA',\tB'),\\
\tLa'' :=  (1,\ty-k,\tilde\w''),\quad
&\p\tLa'' =  (1,\ty-k,\tA'',\tB'').
\end{split}
\end{equation*}
Here $\tilde\w'(\tz)$ is the winding number of 
the loop $\tA'-\tB'$ about $\tz\in\C\setminus(\tA'\cup\tB')$ 
and simiarly for $\tilde\w''$. 
Note that $\tLa''$ is the shift of $\tLa_{k+1}$ by $1$. 
The catenation of $\tLa_{k+1}$ and $\tLa'$ 
is the $(\tal,\tbe)$-trace from $0$ to $1$.
Hence it has Viterbo--Maslov index zero, by Lemma~\ref{le:zero}.
and satisfies
\begin{equation}\label{eq:step3a}
m_j(\tLa_{k+1})+m_j(\tLa')=0\qquad
\forall j\in\Z\setminus\{0,1\}.
\end{equation}
Since the catenation of $\tLa'$ and $\tLa''$ is the 
$(\tal,\tbe)$-trace from $\ty-k-1$ to $\ty-k$,
it also has Viterbo--Maslov index zero and satisfies
\begin{equation}\label{eq:step3b}
m_{\ty-k-j}(\tLa')+m_{\ty-k-j}(\tLa'') = 0
\qquad\forall j\in\Z\setminus\{0,1\}. 
\end{equation}
Moreover, by the induction hypothesis, we have 
\begin{equation}\label{eq:step3c}
m_j(\tLa_{k+1})+m_{\ty-k-1-j}(\tLa_{k+1}) = 0
\qquad\forall j\in\Z\setminus\{0\}. 
\end{equation}
Combining the equations~\eqref{eq:step3a}, \eqref{eq:step3b},
and~\eqref{eq:step3c} we find
\begin{eqnarray*}
m_j(\tLa_k)+m_{\ty-k-j}(\tLa_k) 
&=&
m_j(\tLa_{k+1})+m_j(\tLa')+m_j(\tLa'') \\
&&
+\,m_{\ty-k-j}(\tLa_{k+1}) + m_{\ty-k-j}(\tLa') +m_{\ty-k-j}(\tLa'') \\
&=&
m_j(\tLa_{k+1})+m_j(\tLa') \\
&&
+\, m_{\ty-k-j}(\tLa') + m_{\ty-k-j}(\tLa'')  \\
&&
+\, m_{j-1}(\tLa_{k+1}) + m_{\ty-k-j}(\tLa_{k+1})  \\
&=&
0
\end{eqnarray*}
for $j\in\Z\setminus\{0,1\}$. 
For $j=1$ we obtain
\begin{eqnarray*}
m_1(\tLa_k)+m_{\ty-k-1}(\tLa_k) 
&=&
m_1(\tLa_{k+1})+m_1(\tLa')+m_1(\tLa'') \\
&&
+\,m_{\ty-k-1}(\tLa_{k+1}) + m_{\ty-k-1}(\tLa') +m_{\ty-k-1}(\tLa'') \\
&=&
m_1(\tLa_{k+1}) + m_{\ty-k-2}(\tLa_{k+1})  \\
&&
+\, m_0(\tLa_{k+1}) + m_{\ty-k-1}(\tLa_{k+1}) \\
&&
+\, m_{\ty-k-1}(\tLa') + m_1(\tLa')  \\
&=&
2\mu(\tLa_{k+1}) + 2\mu(\tLa') \\
&=&
0.
\end{eqnarray*}
Here the last but one equation follows from 
equation~\eqref{eq:step3c} and Proposition~\ref{prop:maslovC},
and the last equation follows from Lemma~\ref{le:zero}.
Hence $\tLa_k$ satisfies~\eqref{eq:jmk}.
This completes the induction argument for the proof of Step~3.

\medskip\noindent{\bf Step~4.}
{\it Let $\tA,\tB\subset\C$ be as in Lemma~\ref{le:AB}
and let $g\in\Gamma$ such that
$$
g\tx\in\tA\cap\tB,\qquad g\ty\notin\tA\cup\tB.
$$ 
Then~\eqref{eq:gm} holds for $g$ and $g^{-1}$.}

\medskip\noindent
%The proof is by induction and catenation 
%based on Step~2 and Lemma~\ref{le:zero}.  
Since $g\ty\notin\tA\cup\tB$ we have $g\ne\id$.
Since $g\tx\in\tA\cap\tB$ we have $\tal=g\tal$ and $\tbe=g\tbe$.  
Hence $\alpha$ and $\beta$ are noncontractible 
embedded circles, and they are homotopic 
(with some orientation) by Lemma~\ref{le:dbae3}.
Thus we may choose $\pi:\C\to\Sigma$, $\tal$, $\tbe$, $\tLa$ 
as in Step~3.  By assumption there is an integer $k\in\tA\cap\tB$. 
Hence $\tA$ and $\tB$ do not contain any negative integers.
Choose $k_\tA,k_\tB\in\N$ such that
$$
\tA\cap\Z=\left\{0,1,\dots,k_\tA\right\},\qquad
\tB\cap\Z=\left\{0,1,\dots,k_\tB\right\}.
$$
Assume without loss of generality that $k_\tA\le k_\tB$. 
For $0\le k\le k_\tA$ denote by $\tA_k\subset\tA$ and 
$\tB_k\subset\tB$ the arcs from $0$ to $\ty-k$
and consider the $(\tal,\tbe)$-trace
$$
\tLa_k:=(0,\ty-k,\tilde\w_k),\qquad
\p\tLa_k:=(0,\ty-k,\tA_k,\tB_k).
$$
In this case 
$$
\tB_k\cap\Z = \{0,1,\dots,k_\tB-k\}.
$$
As in Step~3, it follows by reverse induction on $k$ 
that $\tLa_k$ satisfies~\eqref{eq:jmk} for every $k$.   
We assume again that~$\ty$ is not an integer. 
(Otherwise~\eqref{eq:jmk} follows from Lemma~\ref{le:zero}).
If $k=k_\tA$ then $j,\ty-j\notin\tA_k$ for every $j\in\N$,
hence it follows from Step~2 that $\tLa_k$ satisfies~\eqref{eq:JMK},
and hence it follows from Lemma~\ref{le:gm} for the 
projection of $\tLa_k$ to the annulus $\C/\Z$ that $\tLa_k$
also satisfies~\eqref{eq:jmk}.  The induction step is verbatim the 
same as in Step~3 and will be omitted. 
This proves Step~4.

\medskip\noindent{\bf Step~5.}
{\it We prove the proposition.}

\medskip\noindent
If both points $g\tx,g\ty$ are contained in $\tA$
(or in $\tB$) then $g=\id$ by Lemma~\ref{le:AB}, and in this 
case equation~\eqref{eq:GM} is a tautology.
If both points $g\tx,g\ty$ are not contained in $\tA\cup\tB$,
equation~\eqref{eq:GM} has been established in Lemma~\ref{le:wg}.
Moreover, we can interchange $\tx$ and $\ty$
or $\tA$ and $\tB$ as in the proof of Step~2.
Thus Steps~1 and~4 cover the case where precisely one 
of the points $g\tx,g\ty$ is contained in $\tA\cup\tB$
while Step~3 covers the case where $g\ne\id$ and 
both points $g\tx,g\ty$ are contained in $\tA\cup\tB$. 
This shows that equation~\eqref{eq:GM} 
holds for every $g\in\Gamma\setminus\{\id\}$.
Hence, by Lemma~\ref{le:gm}, equation~\eqref{eq:gm}
holds for every $g\in\Gamma\setminus\{\id\}$.
This proves Proposition~\ref{prop:gm}.
\end{proof}

\begin{proof}[Proof of Theorem~\ref{thm:maslov}
in the Non Simply Connected Case]
Choose a universal covering 
$\pi:\C\to\Sigma$ and let $\Gamma$, $\tal$, $\tbe$, 
and $\tLa=(\tx,\ty,\tilde\w)$ be as in Proposition~\ref{prop:gm}.  
Then
$$
m_x(\Lambda)+m_y(\Lambda)
-m_\tx(\tLa)-m_\ty(\tLa)
= \sum_{g\ne\id}\left(m_{g\tx}(\tLa) 
+ m_{g^{-1}\ty}(\tLa)\right)
= 0.
$$
Here the last equation follows from Proposition~\ref{prop:gm}.
Hence, by Proposition~\ref{prop:maslovC}, we have
$$
\mu(\Lambda) = \mu(\tLa)
=\frac{m_\tx(\tLa)+m_\ty(\tLa)}{2}
=\frac{m_x(\Lambda)+m_y(\Lambda)}{2}.
$$
This proves~\eqref{eq:maslov} in the case where 
$\Sigma$ is not simply connected.
\end{proof}

%%%%%%%%%%%%%%%%%%%%%%%%%%%%%%%%%%%%%%%%%%%%%%%%
%%%%%%%%%%%%%%%%%%%%%%%%%%%%%%%%%%%%%%%%%%%%%%%%
%%%%%%%%%%%%%%%%%% Appendix A %%%%%%%%%%%%%%%%%%
%%%%%%%%%%%%%%%%%%%%%%%%%%%%%%%%%%%%%%%%%%%%%%%%
%%%%%%%%%%%%%%%%%%%%%%%%%%%%%%%%%%%%%%%%%%%%%%%%

\appendix

\section{The Space of Paths}\label{app:path}\label{app:A}

We assume throughout that $\Sigma$ is a connected oriented smooth
$2$-manifold without boundary and 
$
\alpha,\beta\subset\Sigma
$
are two embedded loops.  Let
$$
    \Om_{\alpha,\beta}
    := \left\{x\in\Cinf([0,1],\Sigma)\,|\,
       x(0)\in\alpha,\,x(1)\in\beta\right\}
$$
denote the space of paths connecting $\alpha$ to $\beta$.

\begin{proposition}\label{prop:dbae}
Assume that $\alpha$ and $\beta$ are not contractible 
and that $\alpha$ is not isotopic to $\beta$.  
Then each component of $\Om_{\alpha,\beta}$ 
is simply connected and hence
$
H^1(\Om_{\alpha,\beta};\R)=0.
$
\end{proposition}

The proof was explained to us by David Epstein~\cite{DBAE}.
It is based on the following three lemmas. 
We identify $S^1\cong\R/\Z$.

\begin{lemma}\label{le:dbae1}
Let $\gamma:S^1\to\Sigma$ be a noncontractible loop 
and denote by
$$
\pi:\tSi\to\Sigma
$$
the covering generated by $\gamma$.
Then $\tSi$ is diffeomorphic to the cylinder.
\end{lemma}

\begin{proof}
By assumption, $\Sigma$ is oriented and has a nontrivial
fundamental group.  By the uniformization theorem,
choose a metric of constant curvature.
Then the universal cover of $\Sigma$ is isometric
to either $\R^2$ with the flat metric or
to the upper half space $\HH^2$ with the
hyperbolic metric.
The $2$-manifold $\tSi$
is a quotient of the universal cover of
$\Sigma$ by the subgroup of the group
of covering transformations generated by
a single element (a translation in the
case of $\R^2$ and a hyperbolic element
of $\PSL(2,\R)$ in the case of $\HH^2$).
Since $\gamma$ is not contractible, this element
is not the identity.  Hence $\tSi$
is diffeomorphic to the cylinder.
\end{proof}

\begin{lemma}\label{le:dbae2}
Let $\gamma:S^1\to\Sigma$ be a noncontractible loop 
and, for $k\in\Z$, define $\gamma^k:S^1\to\Sigma$
by 
$$
\gamma^k(s):=\gamma(ks).
$$
Then $\gamma^k$ is contractible if and only if $k=0$.
\end{lemma}

\begin{proof}
Let $\pi:\tSi\to\Sigma$ be as in Lemma~\ref{le:dbae1}.
Then, for $k\ne0$, the loop $\gamma^k:S^1\to\Sigma$ lifts to
a noncontractible loop in $\tSi$.
\end{proof}

\begin{lemma}\label{le:dbae3}
Let $\gamma_0,\gamma_1:S^1\to\Sigma$ be noncontractible 
embedded loops and suppose that $k_0,k_1$ 
are nonzero integers such that $\gamma_0^{k_0}$ 
is homotopic to $\gamma_1^{k_1}$.  
Then either $\gamma_1$ is homotopic to $\gamma_0$ and $k_1=k_0$
or $\gamma_1$ is homotopic to ${\gamma_0}^{-1}$ and $k_1=-k_0$.
\end{lemma}

\begin{proof}
Let $\pi:\tSi\to\Sigma$ be the covering generated by $\gamma_0$. 
Then ${\gamma_0}^{k_0}$ lifts to a closed curve in $\tSi$ 
and is homotopic to ${\gamma_1}^{k_1}$.
Hence ${\gamma_1}^{k_1}$ lifts
to a closed immersed curve in $\tSi$.
Hence there exists a nonzero integer $j_1$
such that ${\gamma_1}^{j_1}$ lifts
to an embedding $S^1\to\tSi$.
Any embedded curve in the cylinder is either
contractible or is homotopic to a generator.
If the lift of ${\gamma_1}^{j_1}$ were contractible
it would follow that ${\gamma_0}^{k_0}$ is contractible,
hence, by Lemma~\ref{le:dbae2}, $k_0=0$ 
in contradiction to our assumption.
Hence the lift of ${\gamma_1}^{j_1}$
to $\tSi$ is not contractible.
With an appropriate sign of $j_1$ it follows that
the lift of ${\gamma_1}^{j_1}$ is homotopic
to the lift of $\gamma_0$.
Interchanging the roles of $\gamma_0$ and $\gamma_1$,
we find that there exist nonzero integers
$j_0,j_1$ such that
$$
\gamma_0\sim{\gamma_1}^{j_1},\qquad
\gamma_1\sim{\gamma_0}^{j_0}
$$
in $\tSi$.  Hence $\gamma_0$ is homotopic to 
${\gamma_0}^{j_0j_1}$ in the free loop space of $\tSi$.
Since the homotopy lifts to the cylinder $\tSi$
and the fundamental group of $\tSi$
is abelian, it follows that 
$$
j_0j_1=1.
$$
If $j_0=j_1=1$ then $\gamma_1$ is homotopic to $\gamma_0$,
hence $\gamma_0^{k_1}$ is homotopic to ${\gamma_0}^{k_0}$,
hence ${\gamma_0}^{k_0-k_1}$ is contractible, 
and hence $k_0-k_1=0$, by Lemma~\ref{le:dbae2}.
If $j_0=j_1=-1$ then $\gamma_1$ is homotopic to ${\gamma_0}^{-1}$,
hence $\gamma_0^{-k_1}$ is homotopic to ${\gamma_0}^{k_0}$,
hence ${\gamma_0}^{k_0+k_1}$ is contractible, and hence $k_0+k_1=0$,
by Lemma~\ref{le:dbae2}.  This proves Lemma~\ref{le:dbae3}.
\end{proof}

\begin{proof}[Proof of Proposition~\ref{prop:dbae}]
Orient $\alpha$ and $\beta$ and
and choose orientation preserving
diffeomorphisms 
$$
\gamma_0:S^1\to\alpha,\qquad
\gamma_1:S^1\to\beta.
$$
A closed loop in $\Om_{\alpha,\beta}$ gives rise
to a map $u:S^1\times[0,1]\to\Sigma$ such that
$$
u(S^1\times\{0\})\subset\alpha,\qquad
u(S^1\times\{1\})\subset\beta.
$$
Let $k_0$ denote the degree of $u(\cdot,0):S^1\to\alpha$
and $k_1$ denote the degree of $u(\cdot,1):S^1\to\beta$.
Since the homotopy class of a map $S^1\to\alpha$
or a map $S^1\to\beta$ is determined by the degree
we may assume, without loss of generality, that
$$
u(s,0) = \gamma_0(k_0s),\qquad
u(s,1) = \gamma_1(k_1s).
$$
If one of the integers $k_0,k_1$ vanishes, so does the other,
by Lemma~\ref{le:dbae2}.  If they are both nonzero then 
$\gamma_1$ is homotopic to either $\gamma_0$ 
or $\gamma_0^{-1}$, by Lemma~\ref{le:dbae3}.
Hence $\gamma_1$ is isotopic to either $\gamma_0$ 
or $\gamma_0^{-1}$, by~\cite[Theorem~4.1]{EPSTEIN}.
Hence $\alpha$ is isotopic to $\beta$, in contradiction
to our assumption.  This shows that
$$
k_0=k_1=0.
$$
With this established it follows that the map
$
u:S^1\times[0,1]\to\Sigma
$
factors through a map $v:S^2\to\Sigma$
that maps the south pole to $\alpha$ and the north pole
to $\beta$.  Since $\pi_2(\Sigma)=0$ it follows
that $v$ is homotopic, via maps with fixed north and
south pole, to one of its meridians.
This proves Proposition~\ref{prop:dbae}.
\end{proof}

%%%%%%%%%%%%%%%%%%%%%%%%%%%%%%%%%%%%%%%%%%%%%%%%%%
%%%%%%%%%%%%%%%%%%%%%%%%%%%%%%%%%%%%%%%%%%%%%%%%%%
%%%%%%%%%%%%%%%%%%%%%%%%%%%%%%%%%%%%%%%%%%%%%%%%%%
%%%%%%%%%%%%%%%%%%%%%%  References %%%%%%%%%%%%%%%
%%%%%%%%%%%%%%%%%%%%%%%%%%%%%%%%%%%%%%%%%%%%%%%%%%
%%%%%%%%%%%%%%%%%%%%%%%%%%%%%%%%%%%%%%%%%%%%%%%%%%
%%%%%%%%%%%%%%%%%%%%%%%%%%%%%%%%%%%%%%%%%%%%%%%%%%

\end{document}